\documentclass[10pt]{amsart}
\usepackage{amscd}
\usepackage{amsmath}
\usepackage{amssymb}
\usepackage{latexsym}
\usepackage{lscape}
\usepackage{xypic}
\usepackage{hyperref}
\usepackage{url}
\usepackage[pdftex]{color}

\newtheorem{thm}{Theorem}[section]
\newtheorem{prop}[thm]{Proposition}
\newtheorem{lem}[thm]{Lemma}

\newtheorem{lem-def}[thm]{Lemma-Definition}
\newtheorem{cor}[thm]{Corollary}
\theoremstyle{remark}
\newtheorem{ex}[thm]{Example}

\newtheorem{rmk}{Remark}[section]
\theoremstyle{definition}
\newtheorem{dfn}[thm]{Definition}

\numberwithin{equation}{section}

\input xy
\xyoption{all}

\newcommand{\quash}[1]{}  %%Anything in \quash is ignored
\newcommand{\nc}{\newcommand}
\nc{\on}{\operatorname}

\newcommand{\frakD}{{\mathfrak D}}

\newcommand{\frakX}{{\mathfrak X}}

\newcommand{\bA}{{\mathbb A}}
\newcommand{\bB}{{\mathbb B}}
\newcommand{\bC}{{\mathbb C}}

\newcommand{\bF}{{\mathbb F}}
\newcommand{\bG}{{\mathbb G}}

\newcommand{\bL}{{\mathbb L}}

\newcommand{\bN}{{\mathbb N}}

\newcommand{\bP}{{\mathbb P}}
\newcommand{\bQ}{{\mathbb Q}}
\newcommand{\bR}{{\mathbb R}}
\newcommand{\bS}{{\mathbb S}}
\newcommand{\bT}{{\mathbb T}}

\newcommand{\bZ}{{\mathbb Z}}

\newcommand{\mB}{{\mathcal B}}

\newcommand{\mE}{{\mathcal E}}

\newcommand{\mH}{{\mathcal H}}

\newcommand{\mM}{{\mathcal M}}
\newcommand{\mN}{{\mathcal N}}
\newcommand{\mO}{{\mathcal O}}
\newcommand{\mP}{{\mathcal P}}

\newcommand{\mR}{{\mathcal R}}

\newcommand{\mX}{{\mathcal X}}
\newcommand{\mY}{{\mathcal Y}}

\nc{\al}{{\alpha}} \nc{\be}{{\beta}} \nc{\ga}{{\gamma}}
\nc{\ve}{{\varepsilon}} \nc{\Ga}{{\Gamma}} \nc{\la}{{\lambda}}
\nc{\La}{{\Lambda}}

\nc{\ad}{{\on{ad}}}

\nc{\aff}{{\on{aff}}}
\nc{\Aff}{{\mathbf{Aff}}}

\nc{\Bun}{{\on{Bun}}}

\nc{\der}{{\on{der}}}

\nc{\diag}{{\on{diag}}}
\nc{\dR}{{\on{dR}}}
\newcommand{\End}{{\on{End}}}
\nc{\Fl}{{\calF\ell}}
\newcommand{\Gal}{{\on{Gal}}}
\nc{\gr}{{\on{gr}}}
\newcommand{\Gr}{{\on{Gr}}}
\newcommand{\Hom}{{\on{Hom}}}
\nc{\IC}{{\on{IC}}}
\newcommand{\id}{{\on{id}}}
\nc{\Id}{{\on{Id}}}

\nc{\Ind}{{\on{Ind}}}
\nc{\inv}{{\on{Inv}}}
\nc{\Iso}{{\on{Isom}}}

\nc{\Ql}{{\overline{\bQ}_\ell}}
\nc{\res}{{\on{res}}}
\newcommand{\Res}{{\on{Res}}}

\newcommand{\Spec}{{\on{Spec}}}
\nc{\Spa}{{\on{Spa}}}

\nc{\tr}{{\on{tr}}}

\nc{\mmu}{{\mu_\bullet}}

\newcommand{\GL}{{\on{GL}}}
\nc{\GSp}{{\on{GSp}}} \nc{\GU}{{\on{GU}}} \nc{\SL}{{\on{SL}}}
\nc{\SU}{{\on{SU}}} \nc{\SO}{{\on{SO}}}

\nc{\pf}{{\on{pf}}}

\nc{\wGr}{{\widetilde{\Gr}}}

\nc{\B}{{\mathrm{B}}}

\def\limproj{\mathop{\oalign{{\rm lim}\cr
\hidewidth$\longleftarrow$\hidewidth\cr}}}

\title{Rigidity and a Riemann-Hilbert correspondence for $p$-adic local systems}
\author{Ruochuan Liu}\thanks{R.L. is partially supported by NSFC-11571017 and the Recruitment Program of Global Experts of China}
\address{Ruochuan Liu, Beijing International Center for Mathematical Research, Peking University, 5 Yi He Yuan Road, Beijing, 100871, China.}
\email{liuruochuan@math.pku.edu.cn}
\author{Xinwen Zhu}\thanks{X.Z. is partially supported by NSF DMS-1303296/1535464 and a Sloan Fellowship.}
\address{Xinwen Zhu, Department of Mathematics, Caltech, 1200 East California Boulevard, Pasadena, CA 91125.}
\email{xzhu@caltech.edu}\date{February 19, 2016}
\begin{document}
\maketitle

\begin{abstract}
We construct a functor from the category of $p$-adic \'etale local systems on a smooth rigid analytic variety $X$ over a $p$-adic field to the category of vector bundles with an integrable connection on its ``base change to $\B_\dR$", which can be regarded as a first step towards the sought-after $p$-adic Riemann-Hilbert correspondence.
As a consequence, we obtain the following rigidity theorem for $p$-adic local systems on a connected rigid  analytic variety:  if the stalk of such a local system at one point, regarded as a $p$-adic Galois representation, is de Rham in the sense of Fontaine, then the stalk at every point is de Rham. 
Along the way, we also establish some basic properties of the $p$-adic Simpson correspondence. Finally, we give an application of our results to Shimura varieties.
\end{abstract}
\tableofcontents

\section{Introduction}
We begin with some applications of our theory. 

\begin{thm}\label{Main cor}
Let $X$ be a geometrically connected algebraic variety over a number field $E$ and let $\bL$ be a $p$-adic \'etale local system on $X$. Assume that there exists a closed point $x\in |X|$, such that the stalk $\bL_{\bar x}$ of $\bL$ at some geometric point $\bar x$ over $x$, regarded as a $p$-adic Galois representation of the residue field of $x$, is geometric in the sense of Fontaine-Mazur (i.e. it is unramified almost everywhere and is de Rham at $p$), then the stalk $\bL_{\bar y}$ at every closed point $y\in |X|$, regarded as the Galois representation of the residue field of $y$, is geometric.
\end{thm}

Before continuing, let us first make a few remarks.

\begin{rmk}\label{relative FM}
Recall that all $p$-adic representations coming from \'etale cohomology of algebraic varieties are geometric and the Fontaine--Mazur conjecture predicts the converse. In light of the above theorem, it seems reasonable to make the following relative version of the Fontaine--Mazur conjecture. 

\medskip

\noindent\bf Conjecture. \rm Let $\bL$ be an \'etale $\bQ_p$-local system over a geometrically connected algebraic variety $X$ over a number field $E$. Let $\eta$ be the generic point of $X$. If for some $x\in |X|$, $\bL_{\bar x}$ is geometric, then there exists some algebraic variety $Y$ over $\eta$ such that $\bL_{\bar \eta}$ appears as a subquotient of the \'etale cohomology of $Y_{\bar\eta}$ up to Tate twists. 

\medskip

Note that this conjecture holds if the monodromy of $\bL$ is abelian. Indeed, in this case, the geometric monodromy must be finite so after passing to a finite \'etale cover, $\bL$ becomes geometrically constant, i.e. a pullback of an abelian Galois representation (of a finite extension of $E$). As the Fontaine--Mazur conjecture is known in the abelian case (e.g. see \cite{He}), the above conjecture also holds in this case.
\end{rmk}

\begin{rmk}
One may compare the above theorem with the following Deligne's result \cite{De2} (called Principle B in \emph{loc. cit.}): Let $S$ be a smooth connected complex algebraic variety and let $f:X\to S$ be a proper smooth morphism.  Let $\{t_s\}_{s\in S}$ be a family of Hodge classes parametrized by $S$ (i.e. a global section of the relative cohomology $R^{2d}f_*\bQ(d)$ that restricts to a Hodge class in $H^{2d}(X_s,\bQ(d))$ at every $s\in S$). If it is absolute Hodge for one point, then $t_s$ is absolute Hodge for all $s\in S$.
\end{rmk}

Here is a concrete application of Theorem \ref{Main cor}, arising as a discussion with K.-W. Lan. Let $(G,X)$ be a Shimura datum. For a (sufficiently small) open compact subgroup $K\subset G(\mathbb A_f)$, let 
$$\mathrm{Sh}_K(G,X)=G(\mathbb Q)\backslash X\times G(\mathbb A_f)/K$$
be the corresponding Shimura variety. 
Let $V$ be a $\bQ$-rational representation of $G$, which is trivial on $Z_G^s$. Here $Z_G^s$ is the largest anisotropic subtorus in the center of $G$ that is $\bR$-split. Then it is known that $V$ induces a Betti local system $\bL_V$ on $\on{Sh}_K(G,X)$. In addition, the theory of canonical models gives a model of $\mathrm{Sh}_K(G,X)$ (still denoted by the same notation) defined over the reflex field $E\subset \mathbb C$, and for a choice of a prime $p$, a $p$-adic \'etale local system $\mathbb L_{V,p}$ over $\mathrm{Sh}_K(G,X)$ (cf. \cite[\S III.6]{M}). Applying Theorem \ref{Main cor} to special points on Shimura varieties, we obtain the following theorem.

\begin{thm}\label{C: application}
Keep notations and assumptions as above. Then for every closed point $x$ of $\on{Sh}_K(G,X)$, the stalk $(\bL_{V,p})_{\bar{x}}$, regarded as a Galois representation of $\Gal(\overline E/E(x))$, is geometric in the sense of Fontaine-Mazur.
\end{thm}
See Section \ref{S: Shimura} for the proof and related results.

\begin{rmk}
Theorem \ref{C: application} was known if $(G,X)$ is of abelian type (with a few additional assumptions), as in this case $\on{Sh}_K(G,X)$ parameterizes certain abelian motives and $\bL_{V, p}$ is nothing but their $p$-adic realizations. In the general case, it gives an evidence of Deligne's expectation that a Shimura variety (with a few additional assumptions) should be the moduli space of certain motives (particularly in light of the conjecture in Remark \ref{relative FM}) . 
\end{rmk}

Now we turn to our main theory. Let $k$ be a finite extension of $\bQ_p$ and let $\bar k$ be an algebraic closure of $k$ and $\hat{\bar k}$ be its completion.
The main ingredient of Theorem \ref{Main cor} is the following theorem.

\begin{thm}\label{Main thm}
Let $X$ be a geometrically connected  rigid analytic variety over $k$ and let $\bL$ be a $\bQ_p$-local system on the \'etale site $X_{\on{et}}$. If there exists a classical point $x$ of $X$ such that the stalk $\bL_{\bar x}$ of $\bL$ at some geometric point $\bar x$ over $x$, regarded as a $p$-adic representation of the residue field of $x$, is de Rham, then the stalk $\bL_{\bar y}$ is de Rham at every classical point $y$ of $X$. In addition, $\bL_{\bar y}$ has the Hodge-Tate weights (with multiplicity) as $\bL_{\bar x}$'s. 
\end{thm}

\begin{rmk}
It is not difficult to see that neither the crystalline nor the semi-stable version of the statement of Theorem \ref{Main thm} holds. In fact,  one may consider  a family of abelian varieties, parameterized by a smooth connected rigid analytic variety over $\bQ_p$, such that most of the fibers have good reduction at $p$ but some fibers do not have semi-stable reduction. Then the higher direct images of the trivial local system $\bQ_p$ on the family give rise to desired counterexamples on the base. On the other hand, one would expect that a potentially semi-stable version of Theorem \ref{Main thm} holds. That is, if the family is semi-stable at one classical point, then it becomes semi-stable at every classical point after a finite \'etale extension of $X$, or even after a finite extension of the base field.  
\end{rmk}

\begin{rmk}
A $\bQ_p$-local system on $X$ can be thought of as a geometric family of $p$-adic representations parameterized by $X$. 
One may compare Theorem \ref{Main thm} with the relevant results for arithmetic families of $p$-adic representations \cite{BC, Be}. It says that an arithmetic family of $p$-adic representations is de Rham provided it is de Rham at a Zariski-dense subset of classical points of the base, and the corresponding sets of Hodge-Tate weights are uniformly bounded. Moreover, as showed by families of $p$-adic representations arising on eigenvarieties, the uniformly bounded condition is necessary. Therefore, Theorem \ref{Main thm} exhibits a dichotomy between arithmetic families and geometric families. Namely, the latter are surprisingly rigid. 
\end{rmk}

\begin{rmk}
For a prime $\ell\neq p$, Kisin proved a rigidity theorem for $\bZ_\ell$-local systems in the sense that on a scheme of finite type over a non-archimedean field of residual characteristic $p$,  a $\bZ_\ell$-local system is locally constant in the $p$-adic topology \cite{Ki}. 
\end{rmk}

Using resolution of singularities for rigid analytic varieties (cf. \cite{BM}), it is enough to prove Theorem \ref{Main thm} for smooth varieties. In this case, we will deduce it from a version of $p$-adic Riemann-Hilbert correspondence, which now we explain.

Let $\mathrm{B}_{\dR}$ denote Fontaine's de Rham period ring. We consider a sheaf of $\bQ_p$-algebras $\mO_X\hat\otimes\B_\dR$ on $X_{\hat{\bar k}}$ (see \S \ref{S: p-adic RH theorem}
for the precise definition);
it inherits a filtration from the filtration on $\B_\dR$ and a $\B_\dR$-linear derivation from the derivation on $\mO_X$. The ringed space $(X_{\hat{\bar k}},\mO_X\hat\otimes\B_\dR)$ is denoted by $\mX$, which can be thought of as the ``generic fiber of a canonical lifting of $X_{\hat{\bar k}}$ to $\B_\dR^+$".  We have the following theorem, which can be regarded as a first step towards the long sought-after Riemann-Hilbert correspondence on $p$-adic varieties.

\begin{thm}\label{T: geom RH}(See Theorem \ref{T: p-adic RH} for the full and precise statements.)
Let $X$ be a smooth rigid analytic variety over $k$. Then there is a tensor functor $\mR\mH$ from the category of $\bQ_p$-local systems on $X_{\on{et}}$ to the category of vector bundles on $\mX$ (i.e. certain finite locally free $\mO_X\hat\otimes\B_\dR$-modules), equipped with a semi-linear action of $\Gal(\bar k/k)$, and with a filtration and an integrable connection that  satisfy Griffiths transversality. The functor $\mR\mH$ is compatible with pullback along arbitrary morphisms and (under certain conditions) is compatible with pushforward under smooth proper morphisms.
\end{thm}

\begin{rmk}
In \cite{Be}, the de Rham period sheaf for arithmetic family of Galois representations is defined as the completed tensor product of the structure sheaf of $X$ with $\B_\dR$ over the base field. However, due to the non-existence of the ``canonical embedding $\hat{\bar k}\to\B_\dR$",  the sheaf $\mO_X\hat\otimes\B_\dR$ on $X_{\hat{\bar k}}$ must be defined by a slightly roundabout way. Another crucial difference is that in the geometric situation, there is an additional structure on the vector bundle, namely an integrable connection. As we shall see below, this makes the passage from $\mO_X\hat\otimes\B_\dR$-modules to $\mO_X$-modules simpler and nicer as compared with arithmetic families.
\end{rmk}

We regard Theorem \ref{T: geom RH} as a geometric Riemann-Hilbert correspondence for $p$-adic \'etale local systems. Now let $\varphi: X_{\hat{\bar k}}\to X$ be the natural projection. Since for a local system $\bL$ on $X$ the vector bundle $\mR\mH(\bL)$ admits an action of $\Gal(\bar k/k)$, (informally) one can define 
\[D_{\dR}^i(\bL)=H^i(\Gal(\bar k/k),\varphi_*\mR\mH(\bL)),\quad i\geq 0.\]
See \S\ref{S: arith p-adic RH theorem} for the more precise definition.
Now Theorem \ref{Main thm} (for smooth $X$) follows from the following theorem, which can be regarded as an arithmetic Riemann-Hilbert correspondence.

\begin{thm}\label{main II} (See Theorem \ref{T: p-adic RH for de Rham} for the full and precise statements.)
Let $X$ be a smooth rigid analytic variety over $k$.  Then
\begin{enumerate}
\item[(i)] $D_{\dR}^i$ is a functor from the category of $\bQ_p$-local systems on $X_{\on{et}}$ to the category of vector bundles on $X$ with an integrable connection as above. If $X$ is a point, $D^0_\dR$ coincides with Fontaine's $D_\dR$ functor.
\item[(ii)] The functors $D_{\dR}^i$ commute with arbitrary pullbacks. 
\item[(iii)] If $X$ is connected and there exists a classical point $x$ such that $\bL_{\bar x}$ is de Rham, then there is a decreasing filtration $\on{Fil}$ on $D_\dR^0(\bL)$ by sub-bundles such that $(D_\dR^0(\bL),\nabla_\bL,\on{Fil})$ is the filtered $\mO_X$-module with an integrable connection associated to $\bL$ in the sense of \cite[Definition 7.4]{Sch2}.  In other words, $\bL$ is a de Rham local system in the sense of \cite[Definition 8.3]{Sch2}. 
\end{enumerate}
\end{thm}

\begin{rmk}
(i) Although for most applications (e.g. Theorem \ref{Main thm}) it is enough to make use of $D_\dR^0$, we need to study all $D_\dR^i$ simultaneously in order to prove the above statements (for $D_\dR^0$).

(ii) The local version of $D_\dR^0$ essentially appeared in the work of Brinon \cite{Br2}. 
\end{rmk}

\begin{rmk}
(i) We emphasize that even $(X,\bL)$ is the analytification of an \'etale local system on an algebraic variety over $k$, the vector bundle $D_\dR^0(\bL)$ in the above theorem is a priori analytic. This is consistent with the classical story, where the Riemann-Hilbert correspondence is first established as an equivalence between Betti local systems and complex vector bundles with a flat connection. However, we expect that $D_\dR^0(\bL)$ extends to a logarithmic connection on some nice compactification $\bar X$ and therefore is an algebraic connection with regular singularities. 

(ii) Note that the vector bundle $D_\dR^i(\bL)$ may not have the correct rank in general. So arithmetic Riemann-Hilbert correspondence is only well-behaved for de Rham local systems. In fact, if $X$ is a (not necessarily smooth) algebraic variety, there exists a well-defined category $\on{P}_\dR(X)$ of de Rham perverse sheaves on $X$ and we hope that $D_\dR^0$ can be extended to a functor from $\on{P}_\dR(X)$ to the category of algebraic D-modules on $X$. 

We plan to investigate these extensions in a future work.
\end{rmk}

We explain the idea of the construction of the functor $\mR\mH$.  It is based on the recent progresses in relative $p$-adic Hodge theory \cite{Sch1, Sch2, KL1, KL2}. We will follow notations from \cite{Sch2}.

First recall that in the classical Riemann-Hilbert correspondence, the  functor from local systems to vector bundles with a connection on a complex analytic variety $X$ is given by tensoring with the sheaf of analytic functions
$$\bL\mapsto \bL\otimes_\bC\mO_X,\quad \nabla=1\otimes d.$$ 
In $p$-adic setting, our functor $\mR\mH$  is of the same nature, except that it is well-known that the analytic topology (or even the \'etale topology) in $p$-adic setting is not fine enough and the sheaf of analytic functions on $X$ cannot do the job. The idea then is to consider some finer topology on $X$ and certain period sheaf under this topology as the replacement of $\mO_X$. More precisely, Scholze \cite{Sch2} introduced the 
 the pro-\'etale site $X_{\on{proet}}$ of $X$ as a refinement of the usual \'etale topology on $X$ and therefore admits a natural projection $\nu: X_{\on{proet}}\to X_{\on{et}}$. Let $\nu': X_{\on{proet}}/X_{\hat{\bar k}}\to (X_{\hat{\bar k}})_{\on{et}}$ be the restriction of $\nu$. Every $\bQ_p$-local system $\bL$ then gives rise to (roughly speaking via pullback) a locally constant sheaf of $\bQ_p$-modules on $X_{\on{proet}}$, denoted by $\hat\bL$. In addition, on $X_{\on{proet}}$, there is the de Rham period sheaf $\mO\bB_{\dR}$, which is a module over the structure sheaf with an integrable connection and a decreasing filtration satisfying Griffiths transversality.
Then we define
\[\mR\mH(\bL):=R\nu'_*(\hat\bL\otimes_{\hat\bQ_p}\mO\bB_{\dR}).\]
It follows from the definition that the connection on $\mO\bB_{\dR}$ induces an integrable connection $\nabla_\bL$ on $\mR\mH(\bL)$.
The key point  is then to show that $\mR\mH(\bL)$, which a priori might be a complex of sheaves, is indeed a finitely generated $\mO_{X}\hat\otimes \B_\dR$-module. 

To prove this, we first replace $\mO\bB_{\dR}$ by its $0$th graded piece 
$\mO\bC=\gr^0\mO\bB_{\dR}$. 
Then we need to study 
\[\mH(\bL):= R\nu'_*(\hat\bL\otimes_{\hat\bQ_p} \mO\bC).\]
Note that taking the associated graded of the connection on $\mO\bB_{\dR}$ defines a Higgs field on $\mO\bC$ and therefore a Higgs field $\vartheta_\bL$ on $\mH(\bL)$. It turns out the functor
$$\mH:\bL\mapsto (\mH(\bL),\vartheta_\bL)$$ 
is nothing but (a special case of) the $p$-adic Simpson correspondence, which was first proposed by Faltings \cite{Fa} and recently systematically studied by Abbes-Gros and Tsuji \cite{AG, AGT}\footnote{It is pointed out by Abbes the subtlety to compare our construction with the construction of Abbes-Gros' in \cite{AGT}, although we believe that they are essentially the same.}. We note that these works studied a much more general and subtle theory than what we consider here. But in our special case, the situation is simpler and nicer. For example, we have the following statements.

\begin{thm}\label{main III} (See Theorem \ref{T: p-adic Simpson} for the full and precise statements.)
Let $X$ be a smooth rigid analytic variety over $k$. Then $\mH$ is a tensor functor from the category of $\bQ_p$-local systems on $X_{\on{et}}$ to the category of nilpotent Higgs bundles on $X_{\hat{\bar k}}$. The functor $\mH$ is compatible with pullback along arbitrary morphisms and (under certain conditions) is compatible with pushforward under smooth proper morphisms.
\end{thm}

\begin{rmk}
Our theorem, compared with the $p$-adic Simpson correspondence in \cite{Fa, AGT}, seems to contain the following additional results. 

(i) We show that the Higgs field on $\mH(\bL)$ is nilpotent. This reflects the fact that $\bL$ is defined over $X$ (rather than over $X_{\hat{\bar{k}}}$), and therefore acquires the extra symmetry by the action of the Galois group $\Gal(\bar k/k)$. It is analogous to the following result of Simpson \cite[Corollary 4.2]{Sim}: if $\bL$ is a Betti local system on a complex smooth projective variety that underlies a variation of Hodge structure, then the corresponding Higgs bundle $(M,\vartheta)$ under the (classical) Simpson correspondence is nilpotent.  In fact,  the Galois group of the cyclotomic tower of $k$ should be regarded as the $p$-adic counterpart of the Hodge torus $\bC^\times$ (see  the proof of Lemma \ref{L: unip}). 

(ii) We establish some functoriality of the functor, e.g. its compatibility with pullbacks (which is an important ingredient in the proof of Theorem \ref{Main thm}). 
\end{rmk}

Let us quickly describe the organization of the paper. Our paper is in fact written in the reverse order of the introduction.
We will first discuss the $p$-adic Simpson correspondence in Section \ref{S: p-adic Simpson} and then study $p$-adic Riemann-Hilbert correspondence in Section \ref{S: p-adic RH}. Finally in Section \ref{S: application}, we prove Theorem \ref{Main cor} and give the application to Shimura varieties.

\medskip

\noindent\bf Notations and conventions. \rm In the paper, $k$ denotes a $p$-adic field and $\hat{\bar k}$ a completed algebraic closure of $k$. Let $k_\infty=\cup k_m$ be the cyclotomic extension of $k$ in $\bar k$. Let $K\subset\hat{\bar{k}}$ denote a perfectoid field containing $k_\infty$. Let $K^\flat$ be the tilt of $K$. We fix a compatible system of $p$-power roots of unit $\{\zeta_{p^m}\}_{m\geq0}$ in $k_\infty$, which gives $\epsilon=(1,\zeta_p,\ldots)\in \mO_{K^\flat}$. 

We regard rigid analytic varieties over $L$ ($L=k$ or $K$) as adic spaces locally of finite type over $L$. In particular, affinoid spaces are written as $\Spa(A,A^+)$ with $A^+=A^\circ$. If $X=\Spa(A,A^+)$ is affinoid over $k$, we write $A_K=A\hat\otimes_kK$ and $X_K=\Spa(A_K,A_K^+)$ its base change to $K$.
An \'etale morphism between affinoid spaces is called \emph{standard \'etale} if it is a composition of rational localizations and finite \'etale morphisms.
A standard \'etale morphism from a rigid analytic variety $X$ over $L$ to the torus
$$\bT^n=\Spa(L\langle T_1^{\pm 1},\ldots, T_n^{\pm 1}\rangle,\mO_L\langle T_1^{\pm 1},\ldots, T_n^{\pm 1}\rangle)$$ 
is called a toric chart of $X$. 

Let $X$ be a rigid analytic variety over $k$. Following \cite{Sch2}, let $X_{\on{proet}}$ be the pro-\'etale site on $X$, and let 
$$\nu: X_{\on{proet}}\to X_{\on{et}}$$ 
be the natural projection. Let $\hat\bZ_p$ (resp. $\hat\bQ_p$) be the constant sheaf on $X_{\on{proet}}$ associated to $\bZ_p$ (resp. $\bQ_p$), and for a $\bZ_p$-local system $\bL$ (resp.  $\bQ_p$-local system) on $X_{\on{et}}$, let $\hat\bL$ be the $\hat\bZ_p$-module (resp. $\hat\bQ_p$-module) on $X_{\on{proet}}$ associated to $\bL$ (cf. \cite[\S 8.2]{Sch2} \cite[\S 9.1]{KL1}). 
\medskip

\noindent\bf Acknowledgement. \rm The authors thank A. Abbes, P. Colmez, K. S. Kedlaya, K.-W. Lan, P. Scholze and Y. Tian for valuable discussions, and Koji Shimizu for very careful reading of this paper and useful comments.
Parts of the work were done while the first author was staying at California Institute of Technology and the second author was
staying at Beijing International Center for Mathematical Research. The
authors would like to thank these institutions for their hospitality.

\section{The $p$-adic Simpson correspondence}\label{S: p-adic Simpson}
In this section, we prove Theorem \ref{main III}, which (we believe) is part of the $p$-adic Simpson correspondence as first proposed by Faltings (\cite{Fa}) and recently systematically developed by Abbes-Gros and Tsuji (cf. \cite{AG, AGT}). We use the framework systematically developed in \cite{Sch1,Sch2, KL1,KL2}, and therefore work in the world of rigid analytic geometry (in particular we make use of the pro-\'etale topology \`a la Scholze). We note that some results in this section have counterparts in the works of Faltings, Abbes-Gros, but in a different language and in greater generality. Namely, the works of \cite{Fa, AG, AGT} use formal schemes and Faltings topos, and therefore establish the correspondence at the more subtle integral level. We will try our best to make the translation between these two settings explicit in the sequel. 

\subsection{Statement of the theorem}
Let $X$ be an $n$-dimensional smooth rigid analytic variety over $\Spa(k,\mO_k)$. Faltings introduced the notion of generalized representations on $X_{\hat{\bar k}}$ (in the case when $X$ is the rigid generic fiber of a formal scheme over $\mO_k$), which is a generalization of $p$-adic representations of the geometric fundamental group of $X_{\hat{\bar k}}$. Then
the $p$-adic Simpson correspondence of \cite{Fa} is an equivalence of categories between small generalized representations on $X_{\hat{\bar k}}$ and small Higgs bundles on $X_{\hat{\bar k}}$. Here ``small" refers to those objects close to being trivial. Note that the construction of such a correspondence is not completely canonical. It depends on a choice of lifting $X_{\hat{\bar k}}$ to Fontaine's ring $A_2$, and therefore makes globalization of the construction not straightforward. 

We will consider a small part of this story. Namely, we consider those generalized representations coming from genuine \'etale local systems on $X$ (rather than on $X_{\hat{\bar k}}$), but (a priori) without any smallness assumption. It turns out that the situation is much nicer due to the existence of the extra action of $\Gal(\bar k/k)$. We will attach such an \'etale $\bQ_p$-local system $\bL$ on $X$ a nilpotent Higgs bundle on $X_{\hat{\bar k}}$ in a completely canonical way. In particular, it globalizes automatically. Note that we only have one direction functor from local systems on $X_{\on{et}}$ to Higgs bundles on $X_{\hat{\bar k}}$, which is not an equivalence of categories. Indeed, the functor loses information, as can be already seen in the case $X=\Spa(k, \mO_k)$ (see Remark \ref{R: point p-adic Simpson}).

Let us be more precise. 
We denote by $\mO_X=\nu^*\mO_{X_{\on{et}}}$ the structural sheaf on $X_{\on{proet}}$ and
$\hat\mO_X$ the completed version (\cite[\S 4]{Sch2}). Let $\Omega_X=\nu^*\Omega_{X_{\on{et}}}$. Let $\mO\bB_{\dR}$ be the de Rham period sheaf on $X_{\on{proet}}$ (\cite[\S 6]{Sch2} and \cite{Sch3}). This is a sheaf of $\mO_X$-algebras, which admits a decreasing filtration $\on{Fil}^\bullet\mO\bB_{\dR}$ and an integrable connection
\begin{equation}\label{E: connection}
\nabla: \mO\bB_{\dR}\to \mO\bB_{\dR}\otimes_{\mO_X}\Omega_X,
\end{equation}
satisfying Griffiths transversality. The main player of this section 
is the period sheaf 
$$\mO\bC=\gr^0\mO\bB_{\dR}.$$
Recall that by \cite[\S 6]{Sch2}, $\gr^j\mO\bB_{\dR}$ is isomorphic to the $j$th Tate twist $\mO\bC(j)$ of $\mO\bC$. Taking the associated graded connection, we get a Higgs field
\begin{equation}\label{E: Higgs field on A}
\gr(\nabla): \mO\bC\to \mO\bC\otimes\Omega_X(-1).
\end{equation}
\begin{rmk}
Alternatively, the period sheaf $\mO\bC$ can be constructed as follows. Let
\begin{equation}\label{E: faltings ext}
0\to \hat\mO_X\to \mE\to \hat\mO_X\otimes_{\mO_X}\Omega_X(-1)\to 0
\end{equation}
denote the Faltings' extension, which is an extension of locally free $\hat\mO_X$-modules on $X_{\on{proet}}$. Then 
$$\mO\bC=\on{Sym}_{\hat\mO_X}\mE/(1-1)=\underrightarrow\lim_n\on{Sym}^n_{\hat\mO_X}\mE,$$
where the first $1$ is the unit of the symmetric algebra $\on{Sym}_{\hat\mO_X}\mE$ and the second $1\in \hat\mO_X\subset \mE=\on{Sym}^1\mE$.

Note that the local version of $\mO\bC$ first appeared in the work of Hyodo \cite{Hy} (under the notation $S_\infty$). 
\end{rmk}

To simplify the exposition, we assume that the perfectoid field $K$ is the completion of a Galois extension of $k$ (in $\bar{k}$) that contains $k_\infty$, and let $\Gal(K/k)$ denote the corresponding Galois group. 
Let $$\nu': X_{\on{proet}}/X_K=(X_K)_{\on{proet}}\to (X_K)_{\on{et}}$$ denote the projection from the pro-\'etale site of $X_K$ to the \'etale site of $X_K$. 
The following theorem is a precise version of Theorem \ref{main III}.
\begin{thm}\label{T: p-adic Simpson}
\begin{enumerate}
\item[(i)] Let $\bL$ be a $\bQ_p$-local system on $X_{\on{et}}$ of rank $r$. Then $R^i\nu'_*(\hat\bL\otimes\mO\bC)=0$ for $i>0$ and
\[\mH(\bL):=\nu'_*(\hat{\bL}\otimes\mO\bC), \quad \vartheta_\bL=\nu'_*(\gr\nabla): \mH(\bL)\to \mH(\bL)\otimes\Omega_{X_K}(-1)\] 
is a rank $r$ vector bundle on $X_{K}$ together with a nilpotent Higgs field $\vartheta_{\bL}$ and a natural semi-linear $\Gal(K/k)$-action. 

\item[(ii)] There is a canonical isomorphism 
$${\nu'}^*\mH(\bL)\otimes_{\mO_{X_K}}\mO\bC|_{(X_K)_{\on{proet}}}\simeq (\hat\bL\otimes\mO\bC)|_{(X_K)_{\on{proet}}},$$ 
compatible with the Higgs fields on both sides.

\item[(iii)] Let $f: Y\to X$ be a morphism between smooth rigid analytic varieties over $k$ and let $\bL$ be a $\bQ_p$-local system on $X_{\on{et}}$. Then there is a canonical isomorphism 
$$f^*(\mH(\bL),\vartheta_\bL)\simeq (\mH(f^*\bL),\vartheta_{f^*\bL}).$$

\item[(iv)] We have $\mH(\bQ_p)=\mO_{X_K}$, where $\bQ_p$ denotes the constant local system.
There is a canonical isomorphism
\begin{equation}\label{E: tensor Higgs}
(\mH(\bL_1\otimes\bL_2),\vartheta_{\bL_1\otimes\bL_2})\simeq (\mH(\bL_1)\otimes_{\mO_{X_K}}\mH(\bL_2),\vartheta_{\bL_1}\otimes 1+1\otimes\vartheta_{\bL_2}),
\end{equation}
and 
\begin{equation}\label{E: dual Higgs}
(\mH(\bL^\vee),\vartheta_{\bL^\vee})\simeq (\mH(\bL)^\vee,\vartheta_\bL^\vee).
\end{equation}
In other words, $\mH$ is a tensor functor from the category of $\bQ_p$-local systems on $X$ to the tensor category of nilpotent Higgs bundles with a semi-linear $\Gal(K/k)$-action on $X_K$.

\item[(v)] Let $f:X\to Y$ be a smooth proper morphism of rigid analytic varieties over $k$ and $\bL$ be a $\bZ_p$-local system on $X_{\on{et}}$. Assume that $R^if_*\bL$ is a $\bZ_p$-local system on $Y_{\on{et}}$. Then there is a canonical isomorphism
\[(\mH(R^if_*\bL),\vartheta_{R^if_*\bL})\simeq R^if_{\on{Higgs},*}(\mH(\bL)\otimes\Omega_{X/Y}^\bullet,\vartheta_\bL).\]
\end{enumerate}
\end{thm}

A few notations need to be explained. A priori, $\mH(\bL)$ is a sheaf of $\mO_{(X_K)_{\on{et}}}$-modules on $X_K$. Then Part (i) of the theorem asserts that it is locally free of finite rank on $(X_K)_{\on{et}}$. By \cite[Lemma 7.3]{Sch2}, we may therefore regard it as a vector bundle on $X_K$ (equipped with the analytic topology). In Part (ii), the Higgs field on the right hand side is given by \eqref{E: Higgs field on A} and on the left hand side is the tensor product of the Higgs field on $\mH(\bL)$ and on $\mO\bC$ (see the formula in \eqref{E: tensor Higgs}). 
The pullback functor $f^*$ on the left hand side in Part (iii) is the usual pullback of coherent sheaves on rigid analytic spaces (whereas on the right hand side is the usual pullback of local systems). The Higgs field induces a complex of $\mO_{X_K}$-modules on $X_K$,
\[\mH(\bL)\stackrel{\bar\vartheta_\bL}{\to} \mH(\bL)\otimes\Omega_{X/Y}(-1)\stackrel{\bar\vartheta_\bL}{\to} \mH(\bL)\otimes\Omega^2_{X/Y}(-2)\to\cdots,\]
where $\bar\vartheta_\bL$ is the composition of $\vartheta_{\bL}$ followed by the natural projection 
\[
\mH(\bL)\otimes\Omega_X(-1)\to \mH(\bL)\otimes\Omega_{X/Y}(-1),
\]
and $R^if_{\on{Higgs},*}(\mH(\bL),\vartheta_\bL)$ in Part (v) is the $i$th derived pushforward of this complex with the induced Higgs field.

Before proving the theorem, let us make a few observations/remarks.
We have the corollary of Part (v).
\begin{cor}
If $X$ is smooth proper over $k$, then
\[H^i(X_{\hat{\bar k}}, \bL)\otimes \hat{\bar k}\simeq H^i(X_{\hat{\bar k}}, (\mH(\bL)\otimes\Omega_X^\bullet,\vartheta_\bL)).\]
\end{cor}
In particular, if $\bL=\bZ_p$ is constant, then $(\mH(\bL),\vartheta_\bL)=\mO_{X_{\hat{\bar k}}}$ with the trivial Higgs field, and we recover the classical Hodge-Tate decomposition (cf. \cite[Corollary 1.8]{Sch2}).

\begin{rmk}\label{R: point p-adic Simpson}
Note that if $X=\{x\}=\Spa(k,\mO_{k})$ is a point, then $\mH(\bL)=(\bL_{\bar x}\otimes\hat{\bar k})^{\Gal(\hat{\bar k}/K)}$ with the zero Higgs field.

In addition, Part (iii) implies that for a classical point $y$ of $X$, 
$$\mH(\bL)|_{y\times_k{K}}\simeq (\Ind^{\Gal(\bar{k}/k)}_{\Gal(\bar{k}/k(y))}\bL_{\bar y}\otimes\hat{\bar k})^{\Gal(\bar k/K)},$$ where $\bar y$ is a geometric point lying over $y$ and $k(y)$ is the residue field of $y$,
and the fiber of $\mH(\bL)$ at each point of $X_K$ lying over $y$ is isomorphic to $(\bL_{\bar y}\otimes\hat{\bar k})^{\Gal(\hat{\bar k}/Kk(y))}$.
\end{rmk}

\begin{rmk} It appears at the first glance that the assumptions in Theorem \ref{T: p-adic Simpson} are weaker than the standard assumptions needed for the $p$-adic Simpson correspondence as in \cite{Fa, AGT}. However, a closer inspection reveals that it is not the case. 

(i) In both approaches \cite{Fa, AGT}, the construction of the $p$-adic Simpson correspondence relies on a choice of a lifting of (an integral model of) $X_{\hat{\bar k}}$ to Fontaine's ring $A_2(\hat{\bar k})$. We also implicitly make such a choice. Namely, since the map $k\to \hat{\bar k}$ canonically factors through $k\to A_2(\hat{\bar k})[1/p]\to \hat{\bar k}$, the base change of $X$ to $A_2(\hat{\bar{k}})[1/p]$ provides such a canonical lifting (in fact the Faltings extension \eqref{E: faltings ext} is constructed out of this lifting). See also Remark \ref{R: can lifting}.

(ii) We do not make any smallness assumption on $\bL$. But since the corresponding Higgs field is nilpotent, it follows a posteriori that $\bL$ is always small (as a $\bQ_p$-local system).
\end{rmk}

\begin{rmk}
In order to establish the $p$-adic Simpson correspondence for \emph{all} small generalized representations,
Abbes-Gros constructed much more complicated period sheaves (various ``overconvergent" versions of $\mO\bC$) under the name of Higgs-Tate algebras (see  \cite[III.10]{AGT}).  Since in our situation the Higgs fields are always nilpotent, we do not need to make use of these more complicated sheaves (but it remains a question to compare our functor with Abbes-Gros').
\end{rmk}

\begin{rmk}\label{R: HT local system}
Let $\mO\bB_{\on{HT}}:=\gr\mO\bB_{\dR}=\bigoplus_i \mO\bC(i)$. In \cite{Hy}, Hyodo considered (the local version of) the functor 
$$\bL\mapsto \nu_*(\hat\bL\otimes\mO\bB_{\on{HT}}),$$ 
which reduces to the usual $D_{\on{HT}}$ functor when $X$ is a point. He defined Hodge-Tate local systems as those $\bL$ such that $\nu_*(\hat\bL\otimes\mO\bB_{\on{HT}})$ is a vector bundle on $X$ of the expected rank. It would be interesting to find a more explicit characterization of  Hodge-Tate local systems on $X$, similar to Theorem \ref{T: p-adic RH for de Rham} in the sequel.
\end{rmk}

\subsection{Preliminaries}
In this subsection, we establish some preliminary facts that will be used in the sequel.

Recall that an object $U\in X_{\on{proet}}$ is called affinoid perfectoid if it admits a presentation $ U=\limproj_{i\in I} U_i$ with $U_i=\Spa(A_i,A_i^+)$ such that the completed direct limit $(\hat A,\hat A^+)$ of 
$\{(A_i,A^+_i)\}_{i\in I}$ is perfectoid. In this case, we write $\hat{ U}$ for the affinoid perfectoid space $\Spa (\hat A,\hat A^+)$, which is independent of the choice of the presentation.  
\begin{prop}\label{P: coherent-vanishing}
Let $U$ be affinoid perfectoid in $X_{\on{proet}}$. Then for any finite locally free module $\mM$ over $\hat \mO_{X}|_{ U}$, $H^i(X_{\on{proet}}/ U,  \mM)=0$ for all $i>0$. 
\end{prop}
 \begin{rmk}Note that finite locally free $\hat\mO_X|_{X_{\hat{\bar k}}}$-modules correspond to Faltings' generalized representations in our setting.
 \end{rmk}
\begin{proof}
Let $\mO_{\hat{U}}$ be the adic structure sheaf of the affinoid perfectoid space $\hat{U}$ associated to $U$. By \cite[Theorem 9.2.15]{KL1}, $\mM$ is isomorphic to the pullback of some finite locally free $\mO_{\hat{U}}$-module $\hat \mM$. Namely, there exists an isomorphism
\[
\mM(V)\simeq \hat{\mM}(\hat U)\otimes_{\mO(\hat{U})}\hat{\mO}_X(V)
\]  
 for any object $V\in X_{\on{proet}}|_{U}$,  which is functorial in the obvious sense. 

We may choose a finite cover $\{\hat{U}_j\}_{j\in J}$ of $\hat{U}$ by rational subsets so that  $\hat \mM$ is free on each $\hat{U}_j$. By virtue of \cite[Lemma 2.6.5(a)]{KL1}, for any $j\in J$, $\hat{U}_j$ is obtained by pulling back a rational localization of some affinoid adic space appearing in a pro-\'etale presentation of 
$U$. Let $U_j\subset U$ be the pullback of this localization. Then $U_j$ is affinoid perfectoid and $\hat{U}_j$ is isomorphic to the associated affinoid perfectoid space.  In particular, $\mM$ is free on $X_{\on{proet}}/{U_j}$ for each $j\in J$. Using \cite[Proposition 7.13]{Sch1}, we have 
\[
H^i(X_{\on{proet}}/U_{j_0}\times_U\cdots\times_U U_{j_k}, \mM)=0
\] 
for any $\{j_0,\dots,j_k\}\subset J$ and $i>0$. Consequently, we obtain 
\[
H^i(X_{\on{proet}}/U, \mM)=\check{H}^i(\{U_j\}_{j\in J}, \mM),\quad i\geq 0
\]
by the Cech-to-derived spectral sequence. The right hand side in turn is the same as $\check{H}^i(\{\hat U_j\}_{j\in J},\hat \mM)$. On the other hand, by \cite[Theorem 6.3]{Sch1} and \cite[Theorem 2.7.7]{KL1}, the structure sheaf of a perfectoid space satisfies the Tate sheaf property in the sense of \cite[Definition 2.7.6]{KL1}. This implies  $\check{H}^i(\{\hat U_j\}_{j\in J},\hat \mM)=0$
 for all $i>0$. 
 \end{proof}

\begin{cor}\label{C: vanishing for OB}
Let $\bL$ be a $\bQ_p$-local system on $X_{\on{et}}$ and let $U$ be affinoid perfectoid in $X_{\on{proet}}$. 
For any $-\infty\leq a\leq b\leq \infty$, let $\mO\bB^{[a, b]}_{\dR}=\on{Fil}^a\mO\bB_{\dR}/\on{Fil}^{b+1}\mO\bB_{\dR}$. Then 
$$H^i(X_{\on{proet}}/U, \hat\bL\otimes\mO\bB_{\dR}^{[a,b]})=0,\quad i>0.$$ 
\end{cor}
 \begin{proof}
Note that for every $\bQ_p$-local system $\bL$ on $X$, $\hat\bL$ is locally trivial on $X_{\on{proet}}$, and therefore $\hat\bL\otimes\mM$ is a finite locally free $\hat\mO_X$-module if $\mM$ is so.  

First assume that $a=b$ so $\mO\bB^{[a, b]}_{\dR}=\mO\bC(a)$. We can ignore the Tate twist. Recall that we may write $\mO\bC=\underrightarrow\lim \on{Sym}_{\hat\mO_X}^n\mE$,
where $\mE$ is the Faltings' extension. Since $|X|$ is quasi-compact, the site $X_{\on{proet}}$ is coherent by virtue of \cite[Proposition 3.12(vii)]{Sch2}. Thus cohomology commutes with  direct limit of abelian sheaves over $X_{\on{proet}}$. 
We therefore deduce that  $H^i(X_{\on{proet}}/U,\hat\bL\otimes\mO\bC)=0$ by Proposition \ref{P: coherent-vanishing}. By induction on $b-a$, we deduce the corollary when $a,b\neq \infty$. Using \cite[Lemma 3.18]{Sch2} and the coherence of $X_{\on{proet}}$, we can also allow $a,b=\pm\infty$.
\end{proof}

Next, we introduce a base $\mB$ for $(X_K)_{\on{et}}$ that is useful for computations. 
 Namely we consider a subcategory $\mB$ of $(X_K)_{\on{et}}$, whose objects consist of those  \'etale maps to $X_K$ that are the base changes of standard \'etale morphisms $Y\to X_{k'}$ defined over some finite extension $k'$ of $k$ in $K$ where $Y$ is also required to admit a toric chart after some finite extension of $k'$, and whose morphisms are the base changes of  \'etale  morphisms over some finite extension of $k$ in $K$.
 Note that $\mB$ is in fact a small category and we equip $\mB$ with the induced topology from $(X_K)_{\on{et}}$ (\cite[3.1]{Vi}).

\begin{lem}\label{bases top}
The natural map of the associated topoi  $ (X_K)_{\on{et}}^{\sim}\to \mB^\sim$ is an equivalence.
\end{lem}
\begin{proof}According to \cite[Theorem 4.1]{Vi}, it is enough to show that $\mB$ is a full subcategory and every object in $(X_K)_{\on{et}}$ admits a covering by objects from $\mB$. For the second statement, recall that for any $Y\to X$ \'etale and $y\in Y$, there exists a rational subset $U\subset Y$ containing $y$ such that $U\to X$ is standard \'etale (cf. \cite[Proposition 3.1.4]{DV}). Moreover, by \cite[Lemma 2.1.3(i)]{Ber}, any \'etale morphism  $Y \to X_K$ is, locally on $Y$, a base
change of an \'etale morphism $Y'\to X_{k'}=X\times_kk'$ for some finite extension $k'$ of $k$ in $K$. In addition, we can further cover $Y'$ by rational open subsets, each of which admits a toric chart.

It remains to prove that $\mB$ is a full subcategory of $(X_K)_{\on{et}}$. I.e. we need to show that if $Y_K\to Z_K$ is a morphism in $(X_K)_{\on{et}}$ with $Y=\Spa(B, B^+)$, $Z=\Spa(C, C^+)\in X_{\on{et}}$, then it is the base change of a morphism $Y'\to Z'$ over some finite extension $k'$ of $k$ inside $K$. That is to say, the composite $Y_K\to Z_K\to Z_{k'}$ factors through $Y_{k'}$. Note that by the acyclicity of the structure sheaf of an affinoid space \cite[Theorem 8.2.22]{KL1}, it suffices to prove the assertion \'etale locally. Thus using \cite[Lemma 2.1.3(i)]{Ber} and \cite[Proposition 3.1.4]{DV}, we may reduce to the case that $X=\Spa(A, A^+)$ is an affinoid space, $Z\to X$ is a rational localization or a finite \'etale covering. The case of rational embedding is clear: $Y_K$ maps to the rational subset $Z_K\subset X_K$ implies that $Y$ maps to the rational subset $Z\subset X$.  

Now suppose $A\to C$ is faithfully flat and finite \'etale. It amounts to show that the morphism $f: C\to C_K\to B_K$ factors through $B_{k'}$ for some finite extension $k'$ over $k$ inside $K$. Now fix $c\in C$, and let $F(T)\in A[T]$ be the characteristic polynomial of $c$ over $A$.  Choose $\varepsilon>0$ so that for any $\alpha$ in the Shilov boundary of $B_K$ and any root $x\neq f(c)$ of $f(F(T))$ in the completed residue field $\mathcal{H}(\alpha)$, $|x-f(c)|>\varepsilon$. On the other hand, it is straightforward to see that there exists a finite extension $k'$ over $k$ inside $K$ and $b\in B_{k'}$ such that $|f(c)-b|<\varepsilon$. We claim that $f(c)$ is $\Gal(K/k')$-invariant. Granting the claim, we see that $f(c)\in B_K^{\Gal(K/k')}=B_{k'}$ by the Ax-Sen-Tate theorem. This yields the assertion by choosing a finite set of topological generators of $C$. To prove the claim, we apply the argument as per Krasner's lemma. Since $b$ is $\Gal(K/k')$-invariant, we get $|g(f(c))-f(c)|<\varepsilon$ for any $g\in \Gal(K/k')$. Note that $g(f(c))$ is also a root of $f(F(T))$. This forces $g(f(c))=f(c)$ at the Shilov boundary of $B_K$ by our assumption on $\varepsilon$; thus $g(f(c))=f(c)$.
\end{proof} 

\begin{cor}\label{C: vb}
Let $\mM$ be a rule to functorially assign every $(Y=\Spa(B,B^+)\to X_{k'})\in\mB$ a finite projective $B_K$-module $M(Y_K)$ such that for every standard \'etale map $Z=\Spa(C,C^+)\to Y$, the natural  morphism $M(Y_K)\otimes_{B_K}C_K\to M(Z_K)$ is an isomorphism (Note that $(Z\to Y\to X_{k'})\in\mB$). Then $\mM$ defines a vector bundle on $(X_K)_{\on{et}}$.
\end{cor}
\begin{proof}This follows from the above lemma and  Tate's acyclic theorem and Kiehl's glueing theorem.
\end{proof}

\subsection{Proof of Theorem \ref{T: p-adic Simpson}}

This subsection is devoted to the proof of this theorem modulo Proposition \ref{P: local p-adic Simpson}, which will be established in the next subsection.

First, by definition, $\mH^i(\bL):=R^i\nu'_*(\hat\bL\otimes\mO\bC)$ is the sheafification of the presheaf 
$$(Y\in (X_K)_{\on{et}})\mapsto  H^i(X_{\on{proet}}/Y, \hat\bL\otimes\mO\bC).$$
Therefore by Corollary \ref{C: vb}, to prove that $\mH(\bL)=\mH^0(\bL)$ is a vector bundle of rank $r=\on{rk}\bL$ and $\mH^i(\bL)=0$ for $i>0$, it is enough to show that
\begin{enumerate}
\item[(a)]
If $X=\Spa(A,A^+)$ admits a toric chart, $H^0(X_{\on{proet}}/X_K, \hat\bL\otimes\mO\bC)$ is a finite projective $A_K$-module of rank $r$;
\item[(b)]If in addition $Y=\Spa(B, B^+)\in X_{\on{et}}$ is standard \'etale, 
$H^i(X_{\on{proet}}/Y_K, \hat\bL\otimes\mO\bC)=0$ 
for $i>0$, and $$H^0(X_{\on{proet}}/Y_K, \hat\bL\otimes\mO\bC)=H^0(X_{\on{proet}}/X_K, \hat\bL\otimes\mO\bC)\otimes_{A_K}B_K.$$ 
\end{enumerate}
Since these two statements are of local natural, we can in addition assume that $\bL$ is a $\bZ_p$-local system.
We will keep these assumptions in the proof.

For $m\geq 0$, let
\[\bT_m^n=\Spa(k_m\langle T_1^{\pm1/p^m},\ldots, T_n^{\pm 1/p^m}\rangle,\mO_{k_m}\langle T_1^{\pm 1/p^m},\ldots, T_n^{\pm 1/p^m}\rangle),\]
and
\[
X_m=\Spa(A_m, A_m^+)=X\times_{\bT^n} \bT^n_m.
\] 
Let $\tilde\bT_\infty^n=\underleftarrow\lim_m \bT^n_m$, and
$\tilde{X}_\infty=X\times_{\bT^n}\tilde \bT^n_\infty$ be the affinoid perfectoid in $X_{\on{proet}}$ represented by the \emph{relative toric tower} (\cite[Definition 7.2.4]{KL2})
\begin{equation}\label{E: relative-toric-tower}
\cdots\to X_m\to\cdots\to X_1\to X_0=X.
\end{equation}
In particular, let $(\hat A_\infty, \hat{A}_{\infty}^+)$ be the perfectoid affinoid completed direct limit of the $(A_m, A_m^+)$'s, then $\hat{X}_\infty=\Spa(\hat{A}_{\infty}, \hat{A}_{\infty}^+)$ is the affinoid perfectoid space associated to $\tilde{X}_\infty$. 

Let $X_{K,m}=\Spa(A_{K,m}, A^+_{K,m})$ be the base change of $X_m$ to $K$ over $k_m$. Let $\tilde X_{K,\infty}$ be the affinoid perfectoid represented by the the \emph{toric tower} (\cite[Definition 7.1.4]{KL2})
\begin{equation}\label{E: toric-tower}
\cdots\to X_{K, m}\to\cdots\to X_{K,1}\to X_{K,0}=X_K,
\end{equation}
and $\hat{X}_{K,\infty}=\Spa(\hat A_{K,\infty}, \hat A^+_{K,\infty})$ be the associated affinoid perfectoid space. The Galois cover $\tilde{X}_{K,\infty}/X$ has Galois group $\Ga$, which fits into a splitting exact sequence
\[
1\to \Ga_{\on{geom}}\to\Ga\to \Gal(K/k)\to 1,
\]
where $\Ga_{\on{geom}}\simeq\bZ_p(1)^n$ is the Galois group of $\tilde X_{K,\infty}$ over $X_K$ such that the $i$th generator $\ga_i\in \Ga_{\on{geom}}$ acts on $T^{1/p^m}_i$ via multiplication by $\zeta_{p^m}$ and on $T_j^{1/p^m}$ as identity for $j\neq i$. 
Then $\Gal(K/k)$ acts on $\Ga_{\on{geom}}$ via the $p$-adic cyclotomic character $$\chi: \Gal(K/k)\to \bZ_p^\times.$$

Similarly, if $f: Y\to X$ is standard \'etale, we define $Y_m=\Spa(B_m, B_m^+)$, $Y_{K, m}=\Spa(B_{K,m}, B_{K,m}^+)$, $\tilde Y_\infty$, $\hat{Y}_{\infty}=\Spa(\hat{B}_{\infty}, \hat{B}_{\infty}^+)$,  $\tilde Y_{K,\infty}$ and $\hat{Y}_{K,\infty}=\Spa(\hat{B}_{K, \infty}, \hat{B}_{K, \infty}^+)$ by pulling back along $f$.

 To proceed, we first replace $H^i(X_{\on{proet}}/Y_K, \hat\bL\otimes\mO\bC)$ by a more computable expression. In the rest of the paper, Galois cohomology always means continuous Galois cohomology. 

\begin{lem}\label{L: comparison}
For $i\geq 0$, the natural map
\[
H^i(\Ga_{\on{geom}}, (\hat\bL\otimes\mO\bC)(\tilde Y_{K,\infty}))\to H^i(X_{\on{proet}}/Y_K, \hat\bL\otimes\mO\bC)
\] 
is an isomorphism.
\end{lem}
\begin{proof}
Using Corollary \ref{C: vanishing for OB}, we may proceed as in the proof of \cite[Lemma 5.6]{Sch2}. Let $\tilde{Y}^{k/Y_K}_{K, \infty}$ be the $k$-fold fiber product of $\tilde{Y}_{K, \infty}$ over $Y_K$.  As $\tilde Y_{K, \infty}\to Y_K$ is a Galois cover with Galois group $\Ga_{\on{geom}}$, $\tilde{Y}^{k/Y_K}_{K,\infty}$ is isomorphic to $\tilde{Y}_{K,\infty}\times \Ga_{\on{geom}}^{k-1}$. By \cite[Lemma 3.16]{Sch2}, it follows that
\[
H^i(X_{\on{proet}}/\tilde{Y}^{k/Y}_{K,\infty}, \hat\bL\otimes\mO\bC)\simeq 
\Hom_{\mathrm{cont}}(\Gamma_{\on{geom}}^{k-1}, H^i(X_{\on{proet}}/\tilde Y_{K,\infty}, \hat\bL\otimes\mO\bC))
\]
for all $k\geq 1$ and $i\geq 0$. Now applying the Cartan-Leray spectral sequence to the Galois cover $\tilde{Y}_{K,\infty}\to Y_K$ and using the fact that $H^i(X_{\on{proet}}/\tilde Y_{K,\infty}, \hat\bL\otimes\mO\bC)=0$ for $i>0$,  we conclude that 
\[
H^i(X_{\on{proet}}/Y_K, \hat\bL\otimes\mO\bC)\simeq
H^i(\Gamma_{\on{geom}}, (\hat\bL\otimes\mO\bC)(\tilde Y_{K,\infty}))
\] 
for all $i\geq0$, yielding the desired isomorphism. 
\end{proof}
Let us write 
$$\mM=\hat\bL\otimes\hat\mO_X.$$ Then $\mM(\tilde Y_{K,\infty})$ is a finite projective $\hat{B}_{K,\infty}$-module equipped with a semi-linear $\Ga$-action. 
According to \cite[\S 6]{Sch2}, over $X_{\on{proet}}/\tilde{Y}_{K}$, there is an isomorphism
\begin{equation}\label{E: isom}
\mO\bC=\gr^0\mO\bB_{\dR}\simeq \hat\mO_{Y_K}[V_1,\ldots,V_d],
\end{equation}
where $V_i=t^{-1}\log([T_i^\flat]/T_i)$ and $t=\log([\epsilon])$, and the $s$th generator $\ga_s$ of $\Ga_{\on{geom}}$ acts on $V_t$ by $\ga_s(V_t)=V_t+\delta_{st}$ (\cite[lemma 6.17]{Sch2}). 
Therefore,
\[
(\hat\bL\otimes \mO\bC)(\tilde Y_{K,\infty})\simeq \mM(\tilde Y_{K,\infty})[V_1,\ldots,V_n].
\]  

Now we may replace $H^i(X_{\on{proet}}/Y_K, \hat\bL\otimes\mO\bC)$ by $H^i(\Ga_{\on{geom}}, \mM(\tilde Y_{K,\infty})[V_1,\ldots,V_n])$ in  (a) and (b). The key statements that allow one to calculate them are as follows. We set $B_{k_m}=B\otimes_k k_m$.
\begin{prop}\label{P: local p-adic Simpson}
There exists a unique finite projective $B_K$-submodule $M_K(Y)$ of $\mM(\tilde Y_{K,\infty})$, which is stable under $\Gamma$, such that
\begin{enumerate}
\item[(i)] $M_K(Y)\otimes_{B_K}\hat{B}_{K,\infty}=\mM(\tilde Y_{K,\infty})$;
\item[(ii)] The $B_K$-linear representation of $\Ga_{\on{geom}}$ on $M_K(Y)$ is unipotent;
\end{enumerate}
In addition, the module $M_K(Y)$ has the following properties:
\begin{enumerate}
\item[(P1)] There exist some positive integer $m_0$ and some finite projective $B_{k_{m_0}}$-submodule $M(Y)$ of $M_K(Y)$ stable under $\Ga$ such that $M(Y)\otimes_{B_{k_{m_0}}}B_K=M_K(Y)$. Moreover, the construction of $M(Y)$ is compatible with base change along standard \'etale morphisms.
\item[(P2)] For every $i\geq 0$, the natural map
$$H^i(\Ga_{\on{geom}}, M_K(Y)) \to H^i(\Gamma_{\on{geom}}, \mM(\tilde Y_{K,\infty}))$$
is an isomorphism.
 \end{enumerate}
\end{prop}

The proof of Proposition \ref{P: local p-adic Simpson} will be given in the next subsection. Granting the proposition, we complete the proof of the theorem in the rest of this subsection. First, taking the logarithmic defines a Higgs field
\begin{equation}\label{E: local Simpson}
(M_K(X),\vartheta: M_K(X)\to M_K(X)\otimes_{A}\Omega^1_{A/k}(-1)).
\end{equation}  
Since the action of $\Gamma_{\on{geom}}$ on $M_K(X)$ is unipotent by Proposition \ref{P: local p-adic Simpson}(ii),  following \cite[I.3.3.2]{AGT}, we normalize $\vartheta$ such that each $\ga_i\in\Ga_{\on{geom}}$ acts on $M_K(X)$ as $\exp(t\vartheta(T_i\frac{d}{dT_i}))$.

One may therefore reformulate Proposition \ref{P: local p-adic Simpson}(P2) by saying that the Higgs cohomology 
\[0\to M_K(X)\stackrel{\vartheta}{\to} M_K(X)\otimes_{A} \Omega_{A/k}(-1)\stackrel{\vartheta\wedge\vartheta}{\to} M_K(X)\otimes_{A}\Omega^2_{A/k}(-2)\to \cdots\]
calculates $H^*(X_{\on{proet}}/X_K, \hat\bL\otimes\hat{\mO}_X)$.
\begin{rmk}\label{R: Sen theory}
The above proposition can be regarded (in our setting) as the local version of the $p$-adic Simpson correspondence for small affine schemes (see \cite{Fa, AGT} for details). It can also be regarded as a relative version of Sen's theory.
Our improvement here is that when $\bL$ is defined over $X$ (rather than $X_K$), the Higgs field is nilpotent (we learned from Abbes that a similar nilpotence statement already appeared in \cite[Proposition 5]{Br1} and \cite[Lemma 9.5]{Ts}).
Our approach is different from \emph{loc. cit.}
\end{rmk}

A priori, the $B_K$-module $M_K(Y)$ depends on the choice of the toric chart $Y\to X\to \bT^n$. But the following lemma says that indeed it does not.

\begin{lem}\label{L: cohomology-computation}
For $i\geq0$, there is a canonical isomorphism 
\begin{equation}\label{E: gamma-geom-cohomology}
H^i(X_{\on{proet}}/Y_K, \hat\bL\otimes\mO\bC)\simeq\begin{cases} M_K(Y)& i=0 \\ 0 & i>0,\end{cases}
\end{equation}
which is $\Gal(K/k)$-equivariant. 
\end{lem}
Note that this statement completes the proof of (a) and (b), and therefore implies that $\mH^i(\bL)=0$ for $i>0$ and $\mH(\bL)=\nu'_*(\hat\bL\otimes\mO\bC)$ is a vector bundle on $X_K$ of the expected rank. 
\begin{proof}
We first need an elementary lemma.
\begin{lem}\label{L: elementary}
Let $W$ be a $\bQ$-vector space with a linear automorphism $\ga:W\to W$. Put
$$W^{\on{gen. inv.}}=\{w\in W\mid (\ga-1)^Nw=0\hspace{1.5mm}\text{for some}\hspace{1.5mm}N\in\bN\},$$ 
i.e. the generalized eigenspace of $\ga$ on $W$ with eigenvalue $1$.
Define an action of $\ga $ on $W[V]$ as $\ga(wV^j)=\ga(w)(V+1)^j$. 
Then the natural map
\[
W[V]\to W,\quad V\mapsto 0
\] 
induces an isomorphism from $H^0(\ga,W[V])=W[V]^{\ga=1}$ to $W^{\on{gen. inv.}}$.
Moreover, if $W=W^{\on{gen. inv.}}$, then $H^1(\ga, W[V])=W[V]/(\ga-1) W[V]$ vanishes. 
\end{lem}
\begin{proof}We consider the basis of $W[V]$ given by $\begin{pmatrix}V \\ i\end{pmatrix}=\dfrac{V(V-1)\cdots(V-i+1)}{i!}$. It is straightforward to see that 
\begin{equation}\label{aux}
(\ga-1)(\sum_{i=0}^{r} w_i \begin{pmatrix}V \\ i\end{pmatrix})= \sum_{i=0}^{r} (\ga(w_{i+1})+\ga(w_i)-w_i) \begin{pmatrix}V \\ i\end{pmatrix},
\end{equation}
where $w_{r+1}=0$. Therefore, we see 
that $\sum_{i=0}^{r} w_i {V\choose i}\in H^0(\ga, W[V])$ if and only if 
\[
w_{i+1}=\gamma^{-1}(w_i-\gamma(w_i)), \qquad i\geq0.
\]
That is, all $w_i$ are uniquely determined by $w_0$ and $(\gamma-1)^{r+1}(w_0)=0$ for some $r$. This yields the first assertion of the lemma.

Now suppose $W=W^{\on{gen. inv.}}$. For $w\in W$, assume that $(\gamma-1)^{r+1}(w)=0$ for some $r\geq0$. A short computation shows that for any $j\geq0$,
\[
w\begin{pmatrix}V \\ j\end{pmatrix}=(\gamma-1)(\sum_{i=1}^r\gamma^{-i}(1-\gamma)^{i-1}(w)\begin{pmatrix}V \\ i+j\end{pmatrix}).
\]
This yields  the second assertion of the lemma. 
\end{proof}

Now we may proceed as in the proof of \cite[Proposition 6.16(i)]{Sch2}. By  Proposition \ref{P:  local p-adic Simpson}(P2) and an argument in \cite{Sch2} (given in the paragraph after Lemma 6.17 of \emph{ibid.}),
the natural map
\[H^i(\Gamma_{\on{geom}}, M_K(Y)[V_1,\dots,V_n])\to H^i(\Gamma_{\on{geom}}, \mM(\tilde Y_\infty)[V_1,\dots,V_n])\]
is an isomorphism. 
Note that for $1\leq s\leq n$, $H^{\bullet}(\bZ_p\gamma_s, M_K(Y)[V_1,\dots,V_s])$ is computed by the complex
\[
M_K(Y)[V_1,\dots,V_s]\stackrel{\gamma_s-1}{\to}M_K(Y)[V_1,\dots,V_s].
\]
By Proposition \ref{P: local p-adic Simpson}(ii), it is straightforward to see that for the action of $\ga_s$,
\[
M_K(Y)[V_1,\dots, V_{s-1}]^{\on{gen. inv.}}= M_K(Y)[V_1,\dots, V_{s-1}].
\] 
Therefore, using Lemma \ref{L: elementary}, we  deduce that 
\[
H^i(\bZ_p\gamma_s, M_K(Y)[V_1,\dots,V_s])\simeq \begin{cases}M_K(Y)[V_1,\dots,V_{s-1}]&i=0\\0&i>0.\end{cases}
\]
Using the Hochschild-Serre spectral sequence, we conclude the desired result by reverse induction. 
\end{proof}

To finish the proof of Part (i), we also need to analyze the Higgs field on $\mH(\bL)$.
\begin{lem}\label{L: two higgs}
Under the identification $\mH(\bL)(X_K)=M_K(X)$ in Lemma \ref{L: cohomology-computation}, $\vartheta_\bL=\vartheta$. In particular, it is nilpotent.
\end{lem}
\begin{proof}
First note that the action of $\Gamma_{\on{geom}}$ on $M_K(X)[V_1,\ldots,V_n]$ is unipotent. Therefore we may endow $M_K(X)[V_1,\ldots,V_n]$ with the Higgs field 
$$\Theta=\vartheta\otimes1+ 1\otimes \sum_i\frac{d}{dV_i}\otimes t^{-1}\frac{dT_i}{T_i}: M_K(X)[V_1,\ldots,V_n]\to M_K(X)[V_1,\ldots, V_n]\otimes \Omega_{X}(-1)$$ 
such that $\ga_i=\exp(t\Theta(T_i\frac{d}{dT_i}))$. Therefore $\nu'_*(\hat\bL\otimes \mO\bC)(X_K)$ is calculated by the Higgs cohomology of 
$(M_K(X)[V_1,\ldots,V_n],\Theta)$. On the other hand, since $V_i=t^{-1}\log\frac{[T_i^\flat]}{T_i}$,
\begin{equation}\label{E: grnabla}
\gr\nabla=-1\otimes t^{-1}\sum_i\frac{d}{dV_i}\otimes \frac{dT_i}{T_i}.
\end{equation} 
Therefore, on the space $\ker \Theta$, $\vartheta=\gr\nabla=\vartheta_\bL$.
\end{proof}

We have finished the proof of Part (i).
Part (ii) is also clear. Namely, the map is induced by the adjunction 
\[{\nu'}^*\nu'_*(\hat\bL\otimes\mO\bC)\to\hat\bL\otimes\mO\bC.\] 
By Proposition \ref{P: local p-adic Simpson}, it becomes an isomorphism after tensoring the source with $\mO\bC$. 

\medskip

Next, we prove Part (iii) of the theorem.  We start with a discussion of maps between period sheaves introduced in \cite[\S 4, \S 6]{Sch2} induced by a morphism of rigid analytic varieties.

Let $f:Z\to X$ be a morphism of smooth rigid analytic varieties over $k$ and let $f_{\on{proet}}:Z_{\on{proet}}\to X_{\on{proet}}$ denote the induced map of pro-\'etale sites\footnote{We use $Z$ instead of $Y$ as in the statement of  the theorem since $Y$ has another meaning according to our previous convention.}. Let $(f^{-1}_{\on{proet}},f_{\on{proet},*})$ denote the pair of adjoint functors (pullback and pushforward) between the corresponding topoi. Note that we reserve the notation $f_{\on{proet}}^*$ for another meaning given in \eqref{E: *pullback}.

Now we construct induced maps between period sheaves. First we have 
$\mO^+_X\to f_{\on{proet},*}\mO^+_Z$
and by completion obtain  
\begin{equation}\label{E: str map 0}
\hat{\mO}^+_X\to f_{\on{proet},*}\hat{\mO}^+_Z,
\end{equation}
 and then by tilting and taking ring of Witt vectors, obtain
$\bA_{\on{inf},X}\to f_{\on{proet},*}\bA_{\on{inf},Z}$. Recall from \cite{Sch2,Sch3} that $\mO\bB^+_{\dR}$ is the sheafification of the presheaf that assigns every affinoid perfectoid $U=\underleftarrow{\lim} U_j\in X_{\on{proet}}$ the direct limit over $j$ of the $\ker\theta$-adic completion of $(\mO^+_X(U_j)\hat\otimes_{W(\kappa)}\bA_{\on{inf}}(U))[\frac{1}{p}]$, where $\kappa$ is the residue field of $k$, the completed tensor product means the $p$-adic completion of the tensor product, and
$$\theta: \mO^+_X(U_j)\hat\otimes_{W(\kappa)}\bA_{\on{inf}}(U)\to \hat\mO^+_X(U_j)$$
is the usual map. Therefore, we have $\mO\bB^+_{\dR,X}\to f_{\on{proet},*}\mO\bB^+_{\dR,Z}$ and
by inverting $t$, 
\begin{equation}\label{E: str map I}
\mO\bB_{\dR,X}\to f_{\on{proet},*}\mO\bB_{\dR,Z},
\end{equation}
which is compatible with filtrations. Finally, taking the associated graded gives
\begin{equation}\label{E: str map II}
 \mO\bC_X\to f_{\on{proet},*}\mO\bC_Z.
\end{equation}

The map \eqref{E: str map 0} (after inverting $p$) allows one to define the pullback functor $f_{\on{proet}}^*$ from the category of $\hat\mO_X$-modules to the category of $\hat\mO_Z$-modules as
\begin{equation}\label{E: *pullback}
f^*_{\on{proet}}\mM:=f^{-1}_{\on{proet}}\mM\otimes_{f_{\on{proet}}^{-1}\hat\mO_X}\hat\mO_Z.
\end{equation}
As usual, $f^*_{\on{proet}}$ is the left adjoint of $f_{\on{proet},*}$, when regarded as a functor from the category $\hat\mO_Z$-modules to the category of $\hat\mO_X$-modules.
The following lemma immediately follows from \cite[Theorem 9.2.15]{KL1}.
\begin{lem}\label{L: pullback}
Assume that $\mM$ is a finite locally free $\hat\mO_X$-module. For every affinoid perfectoid $V\to U$ covering $Z\to X$, i.e. $U$ (resp. $V$)  is an affinoid perfectoid in $X_{\on{proet}}$ (resp. in $Z_{\on{proet}}$), and $V\to U\times_XZ$ is a morphism in $Z_{\on{proet}}$, there is a canonical isomorphism
$$(f^*_{\on{proet}}\mM)(V)= \mM(U)\otimes_{\hat\mO_X(U)}\hat\mO_Z(V).$$ 
\end{lem}

Now let $\bL$ be a $\bQ_p$-local system on $X_{\on{et}}$. Note that there is a canonical isomorphism
\[f_{\on{proet}}^*(\hat\bL\otimes\hat\mO_X)\simeq \widehat{f^*\bL}\otimes\hat\mO_Z.\]
In addition, since $\hat\bL$ is locally trivial on $X_{\on{proet}}$, \eqref{E: str map II} induces a map $
\hat\bL\otimes \mO\bC_X\to f_{\on{proet},*}(\widehat{f^*\bL}\otimes \mO\bC_Z)$,
and therefore the adjunction map
\begin{equation}\label{E: str map III}
f_{\on{proet}}^*(\hat\bL\otimes \mO\bC_X)\to \widehat{f^*\bL}\otimes \mO\bC_Z.
\end{equation}
On the other hand, there is a natural adjunction
\begin{equation}\label{E: str map IV}
\nu'_{X,*}(\hat\bL\otimes\mO\bC_X)\to \nu'_{X,*}f_{\on{proet},*}f^{*}_{\on{proet}}(\hat\bL\otimes\mO\bC_X)=  f_*\nu'_{Z,*}f_{\on{proet}}^{*}(\hat\bL\otimes \mO\bC_X).
\end{equation}
Composing them and by adjunction, we obtain the sought-after map
\begin{equation}\label{E: map1}
f^*\mH(\bL)=f^*\nu'_{X,*}(\hat\bL\otimes\mO\bC_X)\to \nu'_{Z,*}(\widehat{f^*\bL}\otimes \mO\bC_Z)=\mH(f^*\bL).
\end{equation}
It remains to show that it is an isomorphism.

The question is local (in \'etale topology) on the source and target so we can assume that both $X=\Spa(A, A^+)$ and $Z=\Spa(C, C^+)$ are affinoid spaces admitting toric charts, and $\bL$ is a $\bZ_p$-local system.
By writing a map of affinoid algebras $A\to C$ as the composition $A\to A\langle x_1,\ldots x_n\rangle\twoheadrightarrow C$, we may assume that $f:Z\to X$ is either a closed embedding or a smooth projection. In either case, we can (after further localization on the source and target) arrange their toric charts to fit into the following (not necessarily Cartesian) commutative diagram
\begin{equation}\label{E: compatible chart}
\begin{CD}
Z@>>> \bT^m\\
@VVV@VVV\\
X@>>> \bT^n,
\end{CD}
\end{equation}
where in the case of closed embedding, $\bT^m\to \bT^n$ is the embedding of the subtorus given by $T_{m+1}=\cdots=T_n=1$ and in the case of smooth projection $\bT^m\to\bT^n$ is the projection given by $T_i\mapsto T_i, 1\leq i\leq n$. In either case, we fix a map $\tilde\bT^m_\infty \to \tilde\bT^n_\infty$ lifting $\bT^m\to \bT^n$. Note that this naturally gives rise to a map $\Ga_Z\to \Ga_X$, where $\Ga_X$ and $\Ga_Z$ are Galois groups for the towers $\tilde{X}_{K,\infty}/X$ and $\tilde{Z}_{K,\infty}/Z$ respectively, and a map $\tilde Z_{K,\infty}\to Z\times_X\tilde X_{K,\infty}$.

Let us write $\mM=\hat\bL\otimes\hat\mO_X$ and $\mN=\widehat{f^*\bL}\otimes\hat\mO_Z$ so $\mN\simeq f^*_{\on{proet}}\mM$. Evaluating it at $\tilde Z_{K,\infty}$ and by Lemma \ref{L: pullback}, we have
\begin{equation}\label{E: Ohat-isomo}
\mM(\tilde X_{K,\infty})\otimes_{\hat A_{K,\infty}}\hat C_{K,\infty}\simeq \mN(\tilde Z_{K,\infty}).
\end{equation}
In addition, under our choice of toric charts and under the isomorphism \eqref{E: isom},
the adjoint of the map \eqref{E: str map III}, evaluated at $\tilde Z_{K,\infty}$ is given by
\[(\mM(\tilde X_{K,\infty})\otimes_{\hat A_{K,\infty}}\hat C_{K,\infty})[V_1,\ldots, V_n]\to \mN(\tilde Z_{K,\infty})[V_1,\ldots, V_m]\]
where $V_i\mapsto V_i$ for $i\leq n$, and $V_i\mapsto 0$ for $i>n$ if $Z\to X$ is a closed embedding. 

Now, it is straightforward to see that the evaluation of  (\ref{E: map1}) at $Z_K$ is the same as the composition 
\begin{equation}\label{E: cohomology-isom}
\begin{split}
H^0(\Ga_{X,\on{geom}}, \mM(\tilde X_{K,\infty})[V_1,\ldots, V_n])\otimes_{A_K}C_K&\to H^0(\Ga_{Z,\on{geom}},(\mM(\tilde X_{K,\infty})\otimes_{A_K} C_K)[V_1,\ldots,V_n]) \\
&\to  H^0(\Ga_{Z,\on{geom}}, (\mM(\tilde X_{K,\infty})\otimes_{\hat A_{K,\infty}}\hat C_{K,\infty})[V_1,\ldots, V_n])\\
&\to H^0(\Ga_{Z,\on{geom}}, \mN(\tilde Z_{K,\infty})[V_1,\ldots, V_m]).
\end{split}
\end{equation}
To see that (\ref{E: cohomology-isom}) is an isomorphism, let 
\[
M_K(X)\subset \mM(\tilde X_{K,\infty})\quad
\text{and}\quad N_K(Z)\subset \mN(\tilde Z_{K,\infty})
\] 
be as in Proposition \ref{P: local p-adic Simpson} for the local systems $\bL$ and $f^*\bL$ respectively.  It  follows  that
$$(M_K(X)\otimes_{A_K}C_K)\otimes_{C_K}\hat{C}_{K,\infty}=(M_K(X)\otimes_{A_K}\hat{A}_{K,\infty})\otimes_{\hat{A}_{K,\infty}}\hat{C}_{K,\infty}\simeq\mM(\tilde X_{K,\infty})\otimes_{\hat{A}_{K,\infty}}\hat{C}_{K,\infty}\simeq\mN(\tilde Z_{K,\infty}).$$
In addition, $\Ga_{\on{geom},Z}$ acts on $M_K(X)\otimes_{A_K}C_K$ through $\Ga_{\on{geom},Z}\to \Ga_{\on{geom},X}$ and therefore acts unipotently. By Proposition \ref{P: local p-adic Simpson}, we deduce that (\ref{E: Ohat-isomo}) induces an isomorphism $$M_K(X)\otimes_{A_K}C_K\to N_K(Z).$$ 
Together with Lemma \ref{L: cohomology-computation} (and its proof), it implies that (\ref{E: cohomology-isom}) can be identified with
\[
\begin{split}
H^0(\Ga_{X,\on{geom}}, M_K(X)[V_1,\ldots, V_n])\otimes_{A_K}C_K&\to H^0(\Ga_{Z,\on{geom}},(M_K(X)\otimes_{A_K} C_K)[V_1,\ldots,V_n]) \\
&\to H^0(\Ga_{Z,\on{geom}}, N_K(Z)[V_1,\ldots, V_m]),
\end{split}
\]
which is an isomorphism (again by Lemma \ref{L: cohomology-computation}). This proves (iii).

\medskip

Next, we prove Part (iv). The first statement follows from
\begin{equation}\label{E: unit pushforward}
\nu'_*\mO\bC\simeq\mO_{X_K},
\end{equation}
as can be easily seen from the above argument. 

Using Part (ii), we have
\[
{\nu'}^*(\mH(\bL_1)\otimes\mH(\bL_2))\otimes_{\mO_{X_K}}\mO\bC  \simeq {\nu'}^*\mH(\bL_1)\otimes_{\mO_{X_K}} \widehat{\bL_2}\otimes\mO\bC
\simeq \widehat{\bL_1\otimes\bL_2}\otimes\mO\bC  \simeq {\nu'}^*\mH(\bL_1\otimes\bL_2)\otimes \mO\bC.
\]
Pushing forward along $\nu'_*$ and using \eqref{E: unit pushforward}, we conclude \eqref{E: tensor Higgs}.

To prove \eqref{E: dual Higgs}, applying \eqref{E: tensor Higgs} to $\bL_1=\bL$ and $\bL_2=\bL^\vee$, we obtain a canonical map 
\begin{equation}\label{E: dual Higgs2}
\mH(\bL^\vee)\to \mH(\bL)^\vee.
\end{equation} 
It remains to prove that this is an isomorphism. It follows from the construction that this map \eqref{E: dual Higgs2} is compatible with the isomorphism in Part (iii). Therefore, to prove that it is an isomorphism, it is enough to assume that $X$ is a point, in which case this is clear since the map is injective (see Remark \ref{R: point p-adic Simpson}).

\begin{rmk}\label{L: external product}
By virtue of Part (iii), \eqref{E: tensor Higgs} is equivalent to the following K\"unneth type formula: 
Let $\bL_i$ be a local system on $X_i$ for $i=1,2$. Then there is a canonical isomorphism 
$$\mH(\bL_1\boxtimes\bL_2)\simeq\mH(\bL_1)\boxtimes\mH(\bL_2),$$
compatible with the Higgs fields on both sides.
Here, as usual $\bL_1\boxtimes\bL_2$ (resp. $\mH(\bL_1)\boxtimes\mH(\bL_2)$) denotes the external tensor product of local systems (resp. vector bundles) on $X_1\times_kX_2$ (resp. on $(X_1)_K\times_K(X_2)_K$). 
\end{rmk}

\medskip

Finally, we prove Part (v). We only give a sketch since granting Part (i) the other ingredients are already in \cite[\S 8]{Sch2}. First, we have an acyclic complex on $X_{\on{proet}}$
\[0\to f^{*}_{\on{proet}}\mO\bC_Y\to \mO\bC_X\stackrel{\gr(\nabla)}{\to} \mO\bC_X\otimes\Omega_{X/Y}^1(-1)\to \mO\bC_X\otimes \Omega_{X/Y}^2(-2)\to\cdots.\]
This can be deduced using \cite[Proposition 8.5]{Sch2} or a direct computation using the charts \eqref{E: compatible chart} and Formula \eqref{E: grnabla}.

Let us tensor this complex with $\hat\bL$ and push it forward to $(Y_K)_{\on{et}}$ in two ways appearing in the following commutative diagram 
\[\begin{CD}
(X_K)_{\on{proet}}@>f_{\on{proet}}>>(Y_K)_{\on{proet}}\\
@V\nu'_X VV@VV\nu'_{Y}V\\
(X_K)_{\on{et}}@>f_{\on{et}}>>(Y_K)_{\on{et}}.
\end{CD}\]
On the one hand, by Part (i), we have the quasi-isomorphism
\begin{equation}\label{E: Hp}
Rf_{\on{et},*}(\mH(\bL)\otimes\Omega^\bullet_{X/Y},\bar{\vartheta}_\bL)=Rf_{\on{et},*} R\nu'_{X,*}(\hat\bL\otimes\mO\bC_X\stackrel{\gr(\nabla)}{\to}\hat\bL\otimes\mO\bC_X\otimes\Omega_{X/Y}(-1)\to\cdots ).
\end{equation}

On the other hand, we have the following lemma.
\begin{lem}
Let $f: X\to Y$ be a smooth proper morphism of smooth rigid analytic varieties. Let $\mM$ be a locally free $\mO_X$-module on $X_{\on{proet}}$. There is a canonical isomorphism
\[ Rf_{\on{proet},*}\mM\otimes_{\mO_Y}\mO\bC_Y\simeq Rf_{\on{proet},*}(\mM\otimes_{\mO_X} f^*_{\on{proet}}\mO\bC_Y).\]
\end{lem}
\begin{proof}By examining \cite[Lemma 8.6]{Sch2}, one finds that one can replace $\mO\bC_X$ by $f^*_{\on{proet}}\mO\bC_Y$ in the argument.
\end{proof}
In particular, since we assume that all $R^if_*\bL$ are local systems,
\begin{equation}\label{E: ai}
\nu'_{Y,*}(\widehat{Rf_*\bL}\otimes\mO\bC_Y)=R\nu'_{Y,*}Rf_{\on{proet},*}(\hat\bL\otimes f^*_{\on{proet}}\mO\bC_Y).
\end{equation}
Putting \eqref{E: Hp} and \eqref{E: ai} together, Part (v) follows.

\subsection{Proof of Proposition \ref{P: local p-adic Simpson}}
To prove Proposition \ref{P: local p-adic Simpson}, we will employ some key technical ingredients developed in \cite{KL2}. 
 By the main results of \cite[\S7]{KL2}, both toric towers and relative toric towers are \emph{locally decompleting} in the sense of \cite[Definition 5.6.2]{KL2}. Roughly speaking, this means that we have an analog of the classical theorem of Cherbonnier-Colmez for $(\varphi,\Gamma)$-modules arising from those towers. More precisely, we have the following result.

\begin{lem}\label{L: sen-tate-decompletion}
For sufficiently large $m$, there exists a finite projective $B_m$-submodule $M_m(Y)$ of $\mM(\tilde{Y}_\infty)$, which is stable under $\Gamma$,   such that 
\[
M_m(Y)\otimes_{B_m}\hat{B}_\infty=\mM(\tilde{Y}_\infty).
\]
Moreover, 
for $i\geq0$,  the natural maps
\[
H^i(\Ga_{\on{geom}}, M_{m}(Y)\otimes_{B_{k_m}}B_K)
\to H^i(\Gamma_{\on{geom}}, \mM(\tilde{Y}_{K,\infty})) 
\]
are isomorphisms.
In addition, the construction of $M_m(Y)$ is compatible with  base change along standard \'etale morphisms.
\end{lem}
Note that the module $M_m(Y)$ in the lemma is not unique. E. g. we may replace $M_m(Y)$ by $M_{m'}(Y)=M_m(Y)\otimes_{B_m}B_{m'}$ for any $m'\geq m$.
\begin{proof}
Let $\psi, \psi'$ denote the base changes of the relative toric tower (\ref{E: relative-toric-tower}) and toric tower (\ref{E: toric-tower}) along the standard \'etale morphism $f: Y\to X$ respectively. By the main results of \cite[\S7]{KL2}, we have that $\psi$ and $\psi'$ are locally decompleting. In particular, one may descend the \'etale $(\varphi,\Gamma)$-module $\tilde{M}_\psi$ (resp. $\tilde{M}_{\psi'}$) over the perfect period ring $\tilde{\textbf{C}}_\psi$ (resp. $\tilde{\textbf{C}}_{\psi'}$) associated to $f^*\hat{\bL}$ (resp. the pullback of $f^*\hat{\bL}$ to the pro-\'etale site of $Y_K$) to an \'etale $(\varphi, \Gamma)$-module $M_\psi$ (resp. $M_{\psi'}$) over the imperfect period ring $\textbf{C}_\psi$ (resp. $\textbf{C}_{\psi'}$)\footnote{In relative $p$-adic Hodge theory, the $\tilde{\textbf{C}}$ and $\textbf{C}$ types of rings refer to relative version of the perfect and imperfect Robba rings in classical $p$-adic Hodge theory respectively.}. 
By their constructions, $\tilde{M}_{\psi'}$ (resp. $M_{\psi'}$) is isomorphic to the base change of $\tilde{M}_\psi$ (resp. $M_{\psi}$) to $\tilde{\textbf{C}}_{\psi'}$ (resp. $\textbf{C}_{\psi'}$).

Moreover, by \cite[Corollary 5.6.5]{KL2},  there exists some $r_0>0$ such that for $0<s\leq r\leq r_0$ and $i\geq0$, the cochain complex computing 
the analytic cohomology group (\cite[Definition 1.3.7]{KL2}) 
\[
H_{\mathrm{an}}^i(\Ga_{\on{geom}}, \tilde{M}_{\psi'}^{[s, r]}/M_{\psi'}^{[s, r]})
\]
is strict exact, where  $\tilde{M}_{\psi'}^{[s, r]}$ and $M_{\psi'}^{[s, r]}$ are associated $\varphi$-bundles over $\tilde{\textbf{C}}^{[s, r]}_{\psi'}$ and $\textbf{C}^{[s, r]}_{\psi'}$ respectively. We fix such an interval $[s, r]$ with $0<s\leq r/p$, and choose some nonnegative integer $m$ so that $p^ms\leq 1\leq p^mr$. Since $\varphi^{-m}(\tilde{M}_{\psi'}^{[s, r]}))=\tilde{M}_{\psi'}^{[p^ms, p^mr]}$, it follows that the cochain complex computing 
\[
H_{\mathrm{an}}^i(\Ga_{\on{geom}}, \tilde{M}_{\psi'}^{[p^ms, p^mr]}/\varphi^{-m}(M_{\psi'}^{[s, r]}))
\]
is strict exact. 

To apply the above results to our context, consider the composition  
\[
\tilde{\textbf{C}}_{\psi'}^{[p^ms, p^mr]}\hookrightarrow\tilde{\textbf{C}}_{\psi'}^{[1,1]}\stackrel{\theta}\to \hat B_{K,\infty}
\quad(\text{resp. $\tilde{\textbf{C}}_\psi^{[p^ms, p^mr]}\hookrightarrow\tilde{\textbf{C}}_\psi^{[1,1]}\stackrel{\theta}\to \hat B_\infty $}).
\] 
By the construction given in \cite{KL2}, the image of $\varphi^{-m}(\textbf{C}_{\psi'}^{[s, r]})$ (resp. $\varphi^{-m}(\textbf{C}_{\psi}^{[s, r]})$) in $\hat{B}_{K,\infty}$  (resp. $\hat{B}_{\infty}$) is $B_{K,m}$ (resp. $B_{m}$), and the base change of $\tilde{M}_{\psi'}^{[p^ms, p^mr]}$ (resp. $\tilde{M}_\psi^{[p^ms, p^mr]}$) to $\hat{B}_{K,\infty}$ (resp. $\hat{B}_\infty$) via $\theta$
is isomorphic to  $(\bL\otimes \hat\mO_Y)(\tilde{Y}_{K,\infty})$ (resp. $(\bL\otimes \hat\mO_Y)(\tilde{Y}_\infty)$).  Moreover, the commutative diagrams
\begin{equation}\label{E: diagram}
\xymatrix{
\varphi^{-m}(\textbf{C}_{\psi'}^{[s, r]})\ar[d]\ar[r]& B_{K, m}\ar[d]\\
\tilde{\textbf{C}}_{\psi'}^{[p^ms, p^mr]}\ar[r]& \hat{B}_{K,\infty},
}
\qquad
\xymatrix{
\varphi^{-m}(\textbf{C}_{\psi}^{[s, r]})\ar[d]\ar[r]& B_m\ar[d]\\
\tilde{\textbf{C}}_{\psi}^{[p^ms, p^mr]}\ar[r]& \hat{B}_\infty
}
\end{equation}
are co-Cartesian.  We set 
\[
M_m(Y)=\varphi^{-m}(M_{\psi}^{[s, r]})\otimes_{\varphi^{-m}(\textbf{C}_{\psi}^{[s, r]})} B_m.
\]
Clearly,  the base change of $M_m(Y)$ to $\hat{B}_\infty$ is isomorphic to $(\hat\bL\otimes \hat\mO_Y)(\tilde{Y}_\infty)$; this yields the first assertion of the proposition. Moreover, it follows that
\[
M_m(Y){\otimes}_{B_{k_m}}B_K=\varphi^{-m}(M_{\psi'}^{[s, r]})\otimes_{\varphi^{-m}(\textbf{C}_{\psi'}^{[s, r]})} B_{K,m}.
\]
We therefore deduce that the cochain complex computing 
\[
H_{\mathrm{an}}^i(\Ga_{\on{geom}}, (\hat\bL\otimes \hat\mO_Y)(\tilde{Y}_{K,\infty})/(M_m(Y){\otimes}_{B_{k_m}}B_K))
\]
is strict exact. Note that by \cite[Theorem 1.3.8]{KL2}, all the analytic cohomolgy groups are isomorphic to the corresponding continuous cohomology groups in our context. This yields the second assertion of the lemma. It remains to show the base change property for $M_m(Y)$. For rational localizations, this is part of the definition of locally decompleting towers. For finite \'etale extensions, this follows from
\cite[Corollary 5.6.7]{KL2}.
\end{proof}

We denote $M_{m}(Y)\otimes_{B_{k_m}}B_K=M_{m}(Y)\hat\otimes_{{k_m}}K$ by $M_{K,m}(Y)$. This is a $B_{K,m}$-module and in particular a $B_K$-module. It follows that 
\[
M_{K,m}(Y)\otimes_{B_{K,m}}\hat B_{K,\infty}=\mM(\tilde Y_\infty)\otimes_{\hat B_\infty}\hat B_{K,\infty}=\mM(\tilde Y_{K,\infty}).
\]
The following lemma is a crucial observation. The argument is a variation of Grothendieck's proof of $\ell$-adic local monodromy theorem.
\begin{lem}\label{L: unip}
The $B_K$-linear representation of $\Ga_{\on{geom}}$ on $M_{K,m}(Y)$ is quasi-unipotent, i.e. there is a finite index subgroup $\Ga'_{\on{geom}}\subset \Ga_{\on{geom}}$ that acts on $M_{K,m}(Y)$ unipotently.
\end{lem}
\begin{proof}
Recall that $M_{K,m}(Y)$ is the base change to $K$ of the $\Ga$-module $M_m(Y)$ and that $\Gal(K/k)$ acts on $\Ga_{\on{geom}}\simeq\bZ_p(1)^d$ via $\chi$. Thus for any $\gamma\in\Ga_{\on{geom}}$ which is sufficiently close to $1$ and $\delta\in \Ga$, a short computation shows that $\delta(\log\gamma)\delta^{-1}=\chi(\delta)\log\gamma$. This implies that all the coefficients of the characteristic polynomial of $\log\gamma$ must vanish because $\chi(\Gal(K/k))$ has finite cokernel in $\bZ_p^\times$. Therefore, $\log \gamma$ is nilpotent.  This clearly implies the lemma.
\end{proof}

According to Lemma \ref{L: unip}, we have a decomposition
\begin{equation}\label{E: decomp}
M_{K,m}(Y)=\bigoplus M_{K,m}(Y)_{\tau},
\end{equation}
where $\tau$ are characters of $\Ga_{\on{geom}}$ of finite order, and 
$$M_{K,m}(Y)_{\tau}=\{m\in M_{K,m}(Y)\mid (\ga-\tau(\ga))^Nm=0 \mbox{ for } N\gg 0\},$$
is the corresponding generalized eigenspace. 
Then being a direct sum decomposition of finite projective $B_K$-modules, each summand is a finite projective $B_K$-module stable under the action of $\Ga_{\on{geom}}$. Possibly replacing $M_m(Y)$ by $M_m(Y)\otimes_{B_m}B_{m'}$ for some large  $m'$, we may assume that the order of every $\tau$ appearing in the above decomposition is less than or equal to $p^m$. So
 for every $\tau$, there is a monomial $T_1^{i_1}\dots T_n^{i_n}$ in $B_{K,m}$ on which $\Ga_{\on{geom}}$ acts via the character $\tau$. Let 
$$M_K(Y)=M_{K,m}(Y)_1$$ 
be the unipotent part, i.e. the summand corresponding to the trivial character. Then we see that $M_K(Y)\otimes_{B_K}B_{K,m}\to M_{K,m}(Y)$ is surjective and therefore is an isomorphism. It follows that 
$$M_K(Y)\otimes_{B_K}\hat{B}_{K,\infty}=(M_K(Y)\otimes_{B_K}B_{K,m})\otimes_{B_{K,m}}\hat{B}_{K,\infty}=M_{K,m}(Y)\otimes_{B_{K,m}}\hat{B}_{K,\infty}=\mM(\tilde Y_{K,\infty}).$$ 
Note that $M_K(Y)$ is in fact stable under the action of $\Gamma$. This gives the construction of $M_K(Y)$ as in the proposition. The uniqueness is clear since if 
\[
M'_K(Y)\subset \mM(\tilde Y_{K,\infty})=M_K(Y)\otimes_{B_K}\hat{B}_{K,\infty}
\] 
is another such $B_K$-module, by considering the action of $\Ga_{\on{geom}}$, it must be contained in $M_K(Y)$ and therefore must coincide with $M_K(Y)$. 

Similarly, we have a decomposition 
\[
M_m(Y)=\bigoplus M_m(Y)_{\tau}.
\]
Note that 
\[
M(Y):=M_m(Y)_1
\] 
is a $\Ga$-stable $B_{k_m}$-submodule of $M_K(Y)$ such that $M(Y)\otimes_{B_{k_m}} B_K=M_K(Y)$. The compatibility with standard \'etale base change follows from Lemma \ref{L: sen-tate-decompletion}. This proves (P1).
 For any $\tau\neq 1$, there exists some $1\leq s\leq n$ such that $\gamma_s-1$ is invertible on $M_{K,m}(Y)_\tau$; this implies that 
\[
H^i(\Ga_{\on{geom}}, M_{K,m}(Y)_\tau)=0
\] 
for all $i\geq0$ by Hochschild-Serre spectral sequence. Thus the natural maps
$$H^i(\Ga_{\on{geom}}, M_K(Y))\to H^i(\Ga_{\on{geom}}, M_{K,m}(Y))$$ 
are isomorphisms. By Lemma \ref{L: sen-tate-decompletion}, (P2) also follows.

\begin{ex}\label{EX: calculation}
We illustrate Proposition \ref{P: local p-adic Simpson} and its proof by the following two examples. Let $X=\bT^1=\Spa(k\langle T^{\pm 1}\rangle,\mO_k\langle T^{\pm 1}\rangle)$ with the toric chart given by the identity map.

(i) Assume that $\bL=\bZ_p$ is the rank one constant local system. Then $\mM(\tilde X_{K,\infty})=\hat{A}_{K,\infty}$ with the natural $\Ga$-action and $M_K(X)=A_K$.

(ii) Assume that $\zeta_{p^m}\in k$. We have a $\bZ/p^m$-torsor $\pi: \bT^1_m\to \bT^1$. Let $\bL=\pi_*\bZ_p$. This is not a small local system in the sense of \cite{Fa}. As a $\Ga_{\on{geom}}$-module,
$$\mM(\tilde X_{K,\infty})\simeq \bigoplus_{\tau:\Ga_{\on{geom}}/p^m\to k^\times}\hat{A}_{K,\infty}\otimes \tau.$$
If we write $\tau(\ga)=\zeta_{p^m}^{a(\ga)}$ for $a(\ga)\in \bZ/p^m$, then
\[M_K(X)=\bigoplus_\tau A_K\cdot T^{-a(\ga)/p^m}\otimes\tau.\]
Note that although $\bL$ is non-trivial, the action of $\Ga_{\on{geom}}$ on $M_K(X)$ is still trivial.
\end{ex}

\section{A $p$-adic Riemann-Hilbert correspondence}\label{S: p-adic RH}
In this section, we discuss several functors from the category of $p$-adic local systems to the category of vector bundles with an integrable connection, which can be regarded as a first step towards the $p$-adic Riemann-Hilbert correspondence. 

\subsection{A geometric Riemann-Hilbert correspondence}\label{S: p-adic RH theorem}
We continue with the notations as in the previous section. So $X$ is a smooth rigid analytic variety over $k$ and $K\subset \hat{\bar k}$ is a perfectoid field containing $k_\infty$. 
We will first deduce a geometric version of the Riemann-Hilbert correspondence from Theorem \ref{T: p-adic Simpson} after introducing some period sheaves on $(X_K)_{\on{et}}$.
By abuse of notations, we use $\B_\dR^+$ and $\B_\dR$ to denote $\bB^+_\dR(K, \mO_K)$ and $\bB_\dR(K, \mO_K)$ respectively (so they are the $\Gal(\hat{\bar k}/K)$-invariants of the classical Fontaine's rings). 
Recall that if $k'/k$ is a finite extension in $K$, there is a canonical map $k'\to \B_{\dR}^+$.

Recall by Lemma \ref{bases top}, giving a sheaf on $(X_K)_{\on{et}}$ is the same as giving a sheaf on $\mB$. Then we define $\mO_X\hat \otimes (\B^+_\dR/t^i)$ by assigning
\begin{equation}\label{E: assignment}
 (Y=\Spa(B,B^+)\to X_{k'})\in \mB \mapsto B\hat\otimes_{k'}(\B^+_{\dR}/t^i),
\end{equation}
and set 
\[
\mO_X\hat{\otimes}\B^+_\dR=\underleftarrow\lim_i \mO_X\hat{\otimes} (\B^+_\dR/t^i)
\] 
and 
\[
\mO_X\hat \otimes \B_\dR=(\mO_X\hat \otimes \B^+_\dR)[t^{-1}].
\]
We have Tate's acyclicity theorem for $\mO_X\hat \otimes (\B^+_\dR/t^i)$. In particular, this verifies that $\mO_X\hat \otimes (\B^+_\dR/t^i)$ is indeed a sheaf. 
\begin{lem}\label{L: acyclic}
If $X=\Spa(A,A^+)$ is affinoid, then 
\[H^i((X_K)_{\on{et}}, \mO_X\hat \otimes (\B^+_\dR/t^i))=\left\{\begin{array}{cc} A\hat\otimes_k(\B^+_{\dR}/t^i)& i=0\\ 0 & i>0\end{array}\right.\]  
\end{lem}
\begin{proof}
In fact, using a Schauder basis of $B$, we have the following exact sequence
\[
0\to B\hat\otimes_{k'}(t^i\B^+_{\dR}/t^{i+1})\to B\hat\otimes_{k'}(\B^+_{\dR}/t^{i+1})\to B\hat\otimes_{k'}(\B^+_{\dR}/t^i)\to0, 
\] 
where $B\hat\otimes_{k'}(t^i\B^+_{\dR}/t^{i+1})\simeq B\hat\otimes_{k'}(\B^+_{\dR}/t)(i)=B_K(i)$. We deduce the claim by the acyclicity of the structure sheaf $\mO_{X_K}$ and induction on $i$.
\end{proof}

We equip $\mO_X\hat \otimes \B^+_\dR$ with the decreasing filtration 
\[
\mathrm{Fil}^i(\mO_X\hat \otimes \B^+_\dR)=t^i\mO_X\hat \otimes \B^+_\dR,
\]
and $\mO_X\hat \otimes \B_\dR$ with the decreasing filtration
\[
\mathrm{Fil}^i(\mO_X\hat \otimes \B_\dR)=\sum_{j\in\mathbb{Z}}t^{-j}\mathrm{Fil}^{i+j}(\mO_X\hat \otimes\B^+_\dR)=t^{-j}\mathrm{Fil}^{i+j}(\mO_X\hat \otimes\B^+_\dR)
\]
for $j\geq -i$. Using Schauder basis of affinoid algebras, it is straightforward to see the following lemma.
\begin{lem}\label{L: graded}
For $i\in\bZ$, we have 
$\gr^i(\mO_X\hat \otimes \B_\dR)\simeq\mO_{(X_K)_{\on{et}}}(i)$. 
\end{lem}

Let $\la: X_{\on{et}}\to X_{\on{an}}$ be the natural projection from the \'etale site to the analytic site of $X$. 
By abuse of notations, we again denote the sheaves $\la_*(\mO_X\hat \otimes (\B_\dR/t^i))$, $\la_*(\mO_X\hat \otimes \B^+_\dR)$ and $\la_*(\mO_X\hat \otimes \B_\dR)$ by $\mO_X\hat \otimes (\B_\dR/t^i)$, $\mO_X\hat \otimes \B^+_\dR$ and $\mO_X\hat \otimes \B_\dR$ respectively.

\begin{prop}\label{P: equivalence}
Let $X=\Spa(A, A^+)$ be an affinoid space over $k$. Then the following categories are naturally equivalent. 
\begin{enumerate}
\item[(i)] The category of finite projective $A\hat\otimes_k(\B_\dR^+/t^i)$-modules (resp. $A\hat\otimes_k\B_\dR^+$-modules).

\item[(ii)] The category of finite locally free $\mO_X\hat\otimes(\B^+_\dR/t^i)$-modules (resp. $\mO_X\hat\otimes\B^+_\dR$-modules) on $(X_K)_{\on{et}}$. 

\item[(iii)] The category of finite locally free $\mO_X\hat\otimes(\B^+_\dR/t^i)$-modules (resp. $\mO_X\hat\otimes\B^+_\dR$-modules) on $(X_K)_{\on{an}}$. 
\end{enumerate}
\end{prop}
\begin{proof}
First observe that for a finite projective $A\hat\otimes_k(\B_\dR^+/t^i)$-module $M$, it is free if and only if $M/tM$ is free over $A\hat\otimes_kK$. In fact, since $t$ is nilpotent, any lift of a basis of $M/tM$ to $M$ is a basis of $M$ over $A\hat\otimes_k(\B_\dR^+/t^i)$. It follows that if $M$ is a finite projective $A\hat\otimes_k\B_\dR^+$-module, then it is free if and only if $M/tM$ is free over $A\hat\otimes_kK$. Consequently, to prove the lemma, it suffices to treat the case of finite projective $A\hat\otimes_k(\B_\dR^+/t^i)$-modules.

$\mathrm{(i)}\Rightarrow\mathrm{(ii), (iii)}$. If $M$ is a finite projective $A\hat\otimes_k(\B_\dR^+/t^i)$-module, by Lemma \ref{L: acyclic}, the presheaf 
\[
\tilde M(U)=M\otimes_{A\hat\otimes_k(\B_\dR^+/t^i)}\mO_X\hat\otimes(\B^+_\dR/t^i)(U)
\]
is an acyclic sheaf on both $(X_K)_{\on{et}}$ and $(X_K)_{\on{an}}$. Moreover, by the above observation, it is locally free. 

$\mathrm{(iii)}\Rightarrow \mathrm{(i)}$. This amounts to show that the sheaf of algebras $\mO_X\hat\otimes(\B^+_\dR/t^i)$ on $(X_K)_{\on{an}}$ satisfies the Kiehl glueing property \cite[Definition 2.7.6]{KL1}. To this end, we may apply \cite[Proposition 2.4.20]{KL1} to our set up. Then it suffices to show that one can glue for any simple Laurent covering $\mM(B_K)=\mM((B_1)_K)\cup \mM((B_2)_K)$ by Berkovich spectrums. Now we employ the formalism of \emph{glueing square} \cite[Definition 2.7.3]{KL1} and \cite[Proposition 2.7.5]{KL1} to conclude. The only nontrivial part is to verify that the map
\[
\mM(B_1\hat\otimes (\B^+_\dR/t^i))\oplus\mM(B_2\hat\otimes (\B^+_\dR/t^i))\to\mM(B\hat\otimes (\B^+_\dR/t^i))
\]
is surjective. But as $t$ is nilpotent, we have $\mM(B_j\hat\otimes (\B^+_\dR/t^i))=\mM((B_j)_K)$, $j=1,2$, and 
$\mM(B\hat\otimes (\B^+_\dR/t^i))=\mM(B_K)$. 

$\mathrm{(ii)}\Rightarrow\mathrm{(i)}$. First note that by \cite[Lemma 2.6.5(a), Proposition 2.6.8]{KL1}, the basis $\mathcal{B}$ is stable in the sense of \cite[Definition 8.2.19]{KL1}. That is, it is closed under rational localizations and finite \'etale extensions. Thus by \cite[Proposition 8.2.20]{KL1}, it reduces to show that one can glue for any rational covering and any morphism which is faithfully flat and finite \'etale. The case of rational coverings is already proved in the previous paragraph. Now suppose $\Spa(B, B^+)\to\Spa(A, A^+)$ is a faithful finite \'etale morphism. By \cite[Lemma 2.2.12]{KL1}, $B\otimes_A(A\hat\otimes(\B^+_{\dR}/t^i))$ is a complete Banach algebra and thus naturally isomorphic to $B\hat\otimes(\B^+_{\dR}/t^i)$. We therefore use faithfully flat descent to conclude. 
\end{proof}

\begin{cor}The pushforward $\la_*$ induces an equivalence of categories between sheaves of $\mO_X\hat \otimes \B^+_\dR$-modules which are locally free of finite rank on $(X_K)_{\on{et}}$ and on $(X_K)_{\on{an}}$.
\end{cor}

By this corollary, we will not specify the topology in the following definition.

\begin{dfn}
We denote the ringed space $(X_{K}, \mO_X\hat \otimes \B^+_\dR)$ by $\mX^+$ and $(X_{K}, \mO_X\hat \otimes \B_\dR)$ by $\mX$. Thus we may regard $\mO_X\hat \otimes \B^+_\dR$ as the structure sheaf $\mO_{\mX^+}$ on $\mX^+$. Similarly we have $\mO_\mX$. By a vector bundle on $\mX^+$ we mean a locally free 
$\mO_{\mX^+}$-module of finite rank, and by a vector bundle on $\mX$ we mean a sheaf of $\mO_\mX$-modules obtained from a vector bundle on $\mX^+$ by extension of scalars. 
 By a filtered vector bundle on $\mX$ we mean a vector bundle $\mE$ on $\mX$ equipped with a decreasing filtration $\mathrm{Fil}^\bullet\mE$ such that  $t^i\mathrm{Fil}^j\mE=\mathrm{Fil}^{i+j}\mE$.
\end{dfn}

\begin{rmk}\label{R: can lifting} 
One may think $\mX^+$ as the base change of $X$ along the canonical embedding $k\to \B^+_\dR$, which provides a canonical lifting of $X_K$ to $\B_\dR^+$. We hope to elaborate the geometric meaning of this construction in the future.
\end{rmk}

By Proposition \ref{P: equivalence}, there is a natural ``base change" functor $\mE\mapsto \mE\hat\otimes_k\B_{\dR}$ from the category of vector bundles on $X$ to the category of vector bundles on $\mX$. This is an exact functor.
In particular, we denote
$$\Omega^j_{\mX^+/\B_\dR^+}:=\Omega^j_{X_{\on{et}}}\hat\otimes_k\B_\dR^+,\quad \Omega^j_{\mX/\B_\dR}:=\Omega^j_{X_{\on{et}}}\hat\otimes_k\B_\dR.$$
One may regard them as sheaves of relative differentials. 
It is straightforward to see that  $\mO_{\mX^+}$ admits a unique continuous $\B^+_\dR$-linear derivation $\mO_{\mX^+}\to\Omega^1_{\mX^+/\B_\dR^+}$ extending the one on $\mO_{X_{\on{et}}}$; it extends to a $\B_\dR$-linear derivation $\mO_{\mX}\to\Omega^1_{\mX/\B_\dR}$ by inverting $t$. 

\begin{dfn}
Let $\mE$ be a vector bundle on $\mX$. By a connection on $\mE$ we mean a $\B_\dR$-linear map
\[
\nabla: \mE\to \mE\otimes_{\mO_{\mX}}\Omega_{\mX/\B_{\dR}}
\]
of sheaves, which satisfies the Leibniz rule with respect to the derivation on $\mO_{\mX}$. The connection $\nabla$ is called integrable if $\nabla^2=0$.  In this case, we have the de Rham complex of $\mE$ defined in the usual way
\[\on{DR}(\mE,\nabla): \mE\stackrel{\nabla}{\to} \mE\otimes_{\mO_{\mX}}\Omega_{\mX/\B_\dR}\stackrel{\nabla}{\to} \mE\otimes_{\mO_{\mX}}\Omega^2_{\mX/\B_\dR}\to\cdots.\]
If $\mE$ is in addition a filtered vector bundle on $\mX$, we say the connection satisfies the Griffiths transversality if
\[\nabla(\on{Fil}^j\mE)\subset \on{Fil}^{j-1}\mE\otimes_{\mO_{\mX^+}}\Omega_{\mX^+/\B_\dR^+}^1.\]
\end{dfn}

Note that if $f:X\to Y$ is a morphism of smooth rigid analytic varieties over $k$, there is a natural map $f^{-1}\mO_{\mY^+}\to \mO_{\mX^+}$ and therefore there is a well-defined pullback functor $f^*$ of vector bundles from $\mY$ to $\mX$. In addition, if $\mE$ is a vector bundle on $\mY$ with an integrable connection $\nabla$, $f^*\mE$ admits a pullback connection in the usual way. Assume that $f$ is smooth. Then we have a short exact sequence 
\[0\to f^*\Omega_{\mY/\B_\dR}\to \Omega_{\mX/\B_\dR}\to \Omega_{X/Y}\hat\otimes_k\B_\dR\to 0.\] Note that the last term can be regarded as the sheaf of relative differentials $\Omega_{\mX/\mY}$.
If $(\mE,\nabla)$ is a vector bundle with an integrable connection on $\mX$, one can form the relative de Rham complex $\on{DR}_{\mX/\mY}(\mE,\nabla)$ as usual.
Then we define
$$Rf_{\dR,*}(\mE,\nabla)=Rf_*(\on{DR}_{\mX/\mY}(\mE,\nabla)).$$

\begin{lem}\label{L: pushforward OBdR} 

For any $-\infty\leq a\leq b\leq \infty$, let $\mO_{\mX}^{[a, b]}=\on{Fil}^a\mO_{\mX}/\on{Fil}^{b+1}\mO_\mX$. 

\begin{enumerate}
\item[(i)] We have $R\nu'_*\mO\bB^{[a, b]}_{\dR}\simeq \mO^{[a,b]}_{\mX}$ compatible with the natural filtrations. 

\item[(ii)] We have $R\nu'_*(\mO\bB_{\dR}\otimes_{\mO_X}\Omega^j_X)\simeq\Omega^j_{\mX/\B_\dR}$ as $\mO_{\mX}$-modules. 
\end{enumerate}
\end{lem}
\begin{proof}

For Part (i), we first construct a morphism $\mO_{\mX^+}\to \nu'_*\mO\bB^+_{\dR}$ of filtered sheaves. For any $(Y=\Spa(B, B^+)\to X_{k'})\in \mB$,  let $U=\underleftarrow\lim_{j\in J}U_j$ be an affinoid perfectoid over $Y_K$. By construction, $\mO\bB^+_{\dR}(U)$ is the direct limit over $j$ of the $\ker\theta$-adic completion of $(\mO_X^+(U_j)\hat{\otimes}_{W(\kappa)}\bA_{\mathrm{inf}}(U))[1/p]$, where $\kappa$ is the residue field of $k$. It is clear to see that 
\[
\mO_{\mX^+}(Y_K)=\underleftarrow\lim_i B\hat\otimes_{k'}(\B_{\dR}^+/t^i)=\underleftarrow\lim_i (B^{+}\hat\otimes W(\mO_{K^\flat}))[1/p]/\xi^i
\]
naturally maps to $\mO\bB^+_{\dR}(U)$. Moreover, for any two affinoid perfectoids $U_1, U_2$ over $Y_K$, 
it is clear to see that the maps $\mO_{\mX^+}(Y_K)\to\mO\bB^+_{\dR}(U_i)$, $i=1,2$, coincide on the overlap $U_1\times_{Y_K} U_2$. Thus we obtain a morphism $\mO_{\mX^+}(Y)\to\mO\bB^+_{\dR}(Y_K)$. It is straightforward to check that these morphisms on sections give rise to a natural morphism 
$\mu: \mO_{\mX^+}\to \nu'_*\mO\bB^+_{\dR}$ on $\mB$, and thus on $(X_{K})_{\on{et}}$, which respects filtrations on both sides. 

Thus $\mu$ induces a morphism $\mO^{[a,b]}_{\mX}\to \nu'_*\mO\bB^{[a, b]}_{\dR}$. We will show that it is an isomorphism and $R^i\nu'_*\mO\bB^{[a, b]}_{\dR}=0$ for $i>0$.  In fact, using results in the previous section, Lemma \ref{L: graded} and by induction on $b-a$, we deduce that for $Y\in\mB$ and $a,b\in\bZ$, $H^i(X_{\on{proet}}/Y_K, \mO\bB^{[a, b]}_{\dR})=0$ for $i>0$, and  
\[
\mu: \mO^{[a,b]}_{\mX}(Y_K)\to \nu'_*\mO\bB^{[a, b]}_{\dR}(Y_K)
\] 
is an isomorphism. This proves (i) when $a,b$ are finite. We conclude the general case by using \cite[Lemma 3.18]{Sch2}  and the coherence of $X_{\on{proet}}$ \cite[Proposition 3.12(vii)]{Sch2}.  

For Part (ii), we may apply the adjunction formula to this situation. The subtlety is that the restriction of $\mO_X$ on $X_K$ is ``smaller"  than the structure sheaf $\mO_{(X_K)_{\on{et}}}$ due to the fact that $X_K$ is actually the ``completion" of the corresponding object in $X_{\on{proet}}$. But this is already enough to deduce that $R^i\nu'_*(\mO\bB_{\dR}\otimes_{\mO_X}\Omega^j_X)=0$ for $i>0$ by using (i). Moreover, for $(Y=\Spa(B, B^+)\to X_{k'})\in \mB$, by adjunction, we have 
\[
\nu'_*(\mO\bB_{\dR}\otimes_{\mO_X}\Omega_X^j)(Y_K)=\nu'_*\mO\bB_{\dR}(Y_K)\otimes_B\Omega^j_B\simeq \Omega^j_{\mX/\B_\dR}(Y_K).
\]
by (i). One easily checks that these isomorphisms on sections give rise to an isomorphism on $\mB$, and thus on $(X_K)_{\on{et}}$, as sheaves of $\mO_{\mX}$-modules.
 \end{proof}

Now we can state a geometric version of the $p$-adic Riemann-Hilbert correspondence.
\begin{thm}\label{T: p-adic RH}
\begin{enumerate}
\item[(i)] Let $\bL$ be a $\bQ_p$-local system on $X_{\on{et}}$. Then 
$R^i\nu'_*(\hat\bL\otimes\mO\bB^{[a,b]}_\dR)=0$
for $i>0$,  and the functor $\mR\mH(\bL):= \nu'_*(\hat\bL\otimes\mO\bB_\dR)$ is a tensor functor from the category of $\bQ_p$-\'etale local system on $X$ to the category of filtered vector bundles on $\mX$, equipped with a semi-linear $\Gal(K/k)$-action, and with an integrable connection
\[\nabla_\bL: \mR\mH(\bL)\to \mR\mH(\bL)\otimes_{\mO_\mX}\Omega_{\mX/\B_\dR}\]
that satisfy the Griffiths transversality.

\item[(ii)] There is a canonical isomorphism 
\[
(\gr^0\mR\mH(\bL),\gr^0(\nabla_\bL))\simeq(\mH(\bL),\vartheta_\bL),
\]
compatible with Higgs fields on both sides.

\item[(iii)] There is a canonical isomorphism 
$${\nu'}^*\mR\mH(\bL)\otimes_{{\nu'}^*\mO_{\mX}}\mO\bB_{\dR}|_{(X_K)_{\on{proet}}}\simeq (\hat\bL\otimes\mO\bB_{\dR})|_{(X_K)_{\on{proet}}},$$ 
compatible with the filtrations and connections on both sides.

\item[(iv)] If $f:X\to Y$ is a morphism of smooth rigid analytic varieties over $k$, then there is a natural isomorphism $f^*(\mR\mH(\bL), \nabla_\bL)\simeq (\mR\mH(f^*\bL), \nabla_{f^*\bL})$.

\item[(v)] Let $f:X\to Y$ be a smooth proper morphism of smooth rigid analytic varieties over $k$, and $\bL$ be a $\bZ_p$-local system on $X_{\on{et}}$. Assume that $R^if_*\bL$ is a $\bZ_p$-local system on $Y$. Then there is a natural isomorphism
\[
(\mR\mH(R^if_*\bL), \nabla_{R^if_*\bL})\simeq R^if_{\dR,*}(\mR\mH(\bL), \nabla_{\bL}).\]
\end{enumerate}
\end{thm}

\begin{proof} 
First, applying Theorem \ref{T: p-adic Simpson} we may proceed as in the proof of Lemma \ref{L: pushforward OBdR} to conclude that 
$R\nu'_*(\hat\bL\otimes\mO\bB_\dR^{[a,b]})=\nu'_*(\hat\bL\otimes\mO\bB_\dR^{[a,b]})$, that
$\nu'_*(\hat\bL\otimes\on{Fil}^0\mO\bB_\dR)$ is a locally free $\mO_{\mX^+}$-module of finite rank, and that $R\nu'_*(\hat\bL\otimes\mO\bB_\dR)=\nu'_*(\hat\bL\otimes\mO\bB_\dR)$ is the extension of scalars of $\nu'_*(\hat\bL\otimes\on{Fil}^0\mO\bB_\dR)$.  
In particular, the filtration on $\mR\mH(\bL)$ defined by $\on{Fil}^a\mR\mH(\bL):=\nu'_*(\hat\bL\otimes\on{Fil}^a\mO\bB_\dR)$ makes $\mR\mH(\bL)$ a filtered vector bundle on $\mX$.

Note that there is an integrable connection
\begin{equation}\label{E: proet connection}
\nabla:\hat\bL\otimes \mO\bB_{\dR}\to \hat\bL\otimes \mO\bB_{\dR}\otimes\Omega^1_X,
\end{equation}
by tensoring the natural integrable connection on $\mO\bB_{\dR}$ \eqref{E: connection}.
Pushing forward via $\nu'$, and by Lemma  \ref{L: pushforward OBdR}, we have
$$\nu'_*(\hat\bL\otimes \mO\bB_{\dR}\otimes \Omega^1_X)=\nu'_*(\hat\bL\otimes \mO\bB_{\dR})\otimes_{\mO_\mX} \Omega^1_{\mX/\B_\dR},$$
and therefore a connection
\[\nabla_\bL: \mR\mH(\bL)\to \mR\mH(\bL)\otimes_{\mO_\mX}\Omega_{\mX/\B_\dR}.\]
That it is integrable and satisfies the Griffiths transversality follows from  the corresponding statements for the connection \eqref{E: proet connection}.  In addition, since $\vartheta_\bL$ is defined as $\nu'_*(\gr(\nabla)) $, Part (ii) also follows. Then arguing as in Theorem \ref{T: p-adic Simpson} (iv), $\mR\mH$ is a tensor functor. The semi-linear action by $\Gal(K/k)$ is clear. We have established Part (i).
 
The map in Part (iii) comes from the adjunction ${\nu'}^*\nu'_*\to \id$ and that it is an isomorphism follows from Part (ii) and Theorem \ref{T: p-adic Simpson} (ii).

To prove Part (iv), first note that the map of period sheaves \eqref{E: str map I} induces a natural map
\begin{equation}\label{E: dR pullback}
f^*\mR\mH(\bL)=f^*\nu'_{X,*}(\hat\bL\otimes\mO\bB_{\dR,X})\to \mR\mH(f^*\bL)=\nu'_{Y,*}(\widehat{f^*\bL}\otimes\mO\bB_{\dR,Y})
\end{equation}
by a similar procedure as before. It remains to prove that it is an isomorphism. But this follows from Part (ii) and Theorem \ref{T: p-adic Simpson} (iii).

Finally  Part (v) follows from the same argument for Theorem \ref{T: p-adic Simpson} (v), with $\mO\bC$ replaced by $\mO\bB_{\dR}$.
\end{proof}

\begin{rmk}
One can reformulate the above theorem using Deligne's notion of $t$-connections. Namely, let
\[\mR\mH^+(\bL)=R\nu'_*(\hat\bL\otimes\mO\bB_{\dR}^{[0,\infty]}).\]
This is a vector bundle on $\mX^+$, equipped with a $\mathrm{B}_\dR^+$-linear connection
$$\nabla^+=t\nabla: \mR\mH^+(\bL)\to \mR\mH^+(\bL)\otimes_{\mO_\mX^+}\Omega_{\mX^+/\mathrm{B}_\dR^+}.$$ Then $\nabla^+$ is a $t$-connection of $\mR\mH^+(\bL)$ in the sense that 
$$\nabla^+(fm)=m\otimes tdf+f\nabla^+(m),\quad (\nabla^+)^2=0.$$ 
Note that its base change along $\mathrm{B}_\dR^+\to\mathrm{B}_\dR$ recovers $(\mR\mH(\bL),\nabla)$ and its base change  along $\mathrm{B}_\dR^+\to K,\ t\mapsto 0$ recovers
 $(\mH(\bL),\vartheta_\bL)$.
 
In fact, this is just part of the full picture. Namely, one should be able to attach a local system $\bL$ on $X$ a variation of $p$-adic twistors\footnote{We learned that L. Fargues independently observed this.}, which roughly speaking is a vector bundle on $X\times \mathrm{FF}$, where $\mathrm{FF}$ is the Fargues--Fontaine curve, equipped with a $t$-connection along $X$ direction. Restricting to the formal neighborhood of the $\infty$-point of $\mathrm{FF}$ then should recover the above theorem.
\end{rmk}

\subsection{An arithmetic Riemann-Hilbert correspondence}\label{S: arith p-adic RH theorem}
We continue with the notations as in the previous subsection. But instead of studying $R\nu'_*(\hat\bL\otimes\mO\bB_\dR)$, we consider 
$$D_\dR^i(\bL):=R^i\nu_*(\hat\bL\otimes\mO\bB_{\dR}).$$
By the Cartan-Leray spectral sequence,  
$$D_\dR^i(\bL)=H^i(\Gal(K/k),\varphi_*\mR\mH(\bL)),$$ 
where $\varphi: X_K\to X$ is the natural projection.
Similar to (and even simpler than) the previous subsection, by pushing forward \eqref{E: proet connection}, we obtain an integrable connection
\[\nabla_\bL: D_\dR^i(\bL)\to D_\dR^i(\bL)\otimes\Omega_{X_{\on{et}}}.\]

\begin{thm}\label{T: p-adic RH for de Rham}
\begin{enumerate}
\item[(i)] The pair $(D_\dR^i(\bL),\nabla_\bL)$ is a vector bundle with an integrable connection on $X$. It vanishes if $i\geq 2$. 

\item[(ii)] If $f: Y\to X$ is a morphism of smooth rigid analytic varieties over $k$, then there is a canonical isomorphism 
$$f^*(D_\dR^i(\bL),\nabla_{\bL})\simeq (D_\dR^i(f^*\bL),\nabla_{f^*\bL}),$$ 
where the pullback on the left hand side is understood as the usual pullback of vector bundles with a connection whereas on the right hand side as the pullback of \'etale local systems. \\

We further assume that $X$ is connected and there exists a classical point $x$ of $X$ such that $\bL_{\bar x}$ is de Rham.

\item[(iii)] The spectral sequence computing $D^i_\dR(\bL)$ associated to the natural filtration of $\mO\bB_{\dR}$ degenerates at $E_1$-term, which induces a decreasing filtration $\on{Fil}$ on $D_\dR^i(\bL)$ by sub-bundles such that the connection satisfies Griffiths transversality with respect to this filtration. In this case the isomorphism in $\mathrm{(ii)}$ respects this filtration.

\item[(iv)] The local system $\bL$ is a de Rham local system in the sense of \cite[Definition 8.3]{Sch2} such that $(D_\dR^0(\bL),\nabla_\bL, \on{Fil})$ is the associated filtered $\mO_X$-module with an integrable connection in the sense of \cite[Definition 7.4]{Sch2}. In addition, $D_{\dR}^1(\bL)\simeq D_{\dR}^0(\bL)$.

\item[(v)] The functor $D_{\dR}=D_\dR^0$ is a tensor functor from the category of de Rham local systems to the category of filtered $\mO_X$-modules with an integrable connection satisfying the Griffiths transversality.
\end{enumerate}
\end{thm}

\begin{rmk}
Note that Parts (iii) and (iv) give a more practical way to check the de Rham property of a local system in the sense of \cite[Definition 8.3]{Sch2}. 
\end{rmk}

\begin{proof}
We start with the proof of Part (i). 
Note that since there exists an integrable connection on $D_\dR^i(\bL)$, it is enough to prove that $D_\dR^i(\bL)$ is a coherent $\mO_{X_{\on{et}}}$-module. Then it follows from the classical argument 
(cf. \cite[\S1.2]{Ke1}) that $D_{\dR}^i(\bL)$ is automatically a vector bundle.

To prove the coherence, we can assume that $X=\Spa(A,A^+)$ admits a toric chart as in the previous section, and let $K=\widehat{k_\infty}$ be the completion of the cyclotomic tower. Note that by definition, $D_{\dR}^i(\bL)$ is the sheafification of the presheaf
\[
Y\mapsto H^i(X_{\on{proet}}/Y, \hat\bL\otimes\mO\bB_\dR)=H^i(\Gal(k_{\infty}/k), \mR\mH(\bL)(Y_K)), \quad Y\in X_{\on{et}}.
\]
Note that $\gr^j\mR\mH(\bL)\simeq \mH(L)(j)$ by Theorem \ref{T: p-adic RH}. Therefore by (a variant of) Corollary \ref{C: vb}, to prove Part (i), it is enough to show that  
\begin{enumerate}
\item[(a)] For $i\geq0$, the cohomology group $H^i(\Gal(k_{\infty}/k), \mH(\bL)(j)(X_K))$ is a finite $A$-module. Moreover, it vanishes if $|j|\gg 0$ or $i\geq2$.
\item[(b)] For any  standard \'etale map $Y=\Spa(B,B^+)\to X$,  the natural base change map
$$H^i(\Gal(k_{\infty}/k), \mH(\bL)(j)(X_K))\otimes_{A}B\to H^i(\Gal(k_{\infty}/k), \mH(\bL)(j)(Y_K))$$
is an isomorphism.
\end{enumerate}
It remains to apply the following lemma.
\begin{lem}\label{L: cohomology-descent}
Let $M$ be a finite $A_{m_0}$-module endowed with a semi-linear continuous $\Gal(k_\infty/k)$-action. If $m\geq m_0$ is sufficiently large, then for any $\gamma\in\Gal(k_\infty/k)$ with $v_p(\chi(\gamma)-1)\geq m$, $\gamma-1$ is continuously invertible on $(M\hat{\otimes}_{k_{m_0}}K)/(M\otimes_{k_{m_0}}k_{m})$. Consequently, for $i\geq0$, the natural map 
\[
H^i(\Gal(k_\infty/k), M\otimes_{k_{m_0}}k_{m})\to H^i(\Gal(k_\infty/k), M\hat{\otimes}_{k_{m_0}}K)
\]
is an isomorphism. 
\end{lem}
\begin{proof}
This is a simple consequence of the Tate-Sen formalism. For $m\in\bN$, let  $X_m$ be the kernel of  Tate's normalized trace map $K\to k_m$. Then we have $X_m\oplus k_m\simeq K$. By the Tate-Sen conditions for the cyclotomic tower (cf. \cite[Proposition 4.1.1]{BC}), there exists $c>0$ such that for sufficiently large $m$, if $\gamma\in\Gal(k_\infty/k)$ satisfying $v_p(\chi(\gamma)-1)\geq m$, then $\gamma-1$ is invertible on $X_m$ and
\[
|(\gamma-1)^{-1}x|\leq c|x|, \qquad x\in X_m.
\]
It follows that $|(\gamma-1)x|\geq \frac{1}{c}|x|$ for $x\in X_m$.

Since $A_{m_0}$ is an affinoid algebra over $k$, it is naturally endowed with a $k$-Banach algebra structure. Moreover, since $M$ is finite over $A_{m_0}$, we may regard it as a Banach module over $A_{m_0}$. Note that it suffices to treat the case that $v_p(\chi(\gamma)-1)=m$. In this case, it remains to show that if $m$ is sufficiently large, then $\gamma-1$ is invertible on $M\hat{\otimes}_{k_{m_0}}X_m$. In fact, using \cite[Lemma 5.2]{Ke2}, if $m$ is sufficiently large, then for any $a\in A_{m_0}$,
\[
|(\gamma-1)a|\leq\frac{1}{2c}|a|.
\]
Using a finite set of generators of $M$ over $A_{m_0}$ and enlarging $m$ if necessary, we may further suppose that 
\[
|(\gamma-1)y|\leq\frac{1}{2c}|y|
\]
for any $y\in M$. Now for $y\in M, x\in X_m$, 
write 
\[
\gamma(y\otimes x)-y\otimes x=(\gamma-1)(y)\otimes \gamma(x)+y\otimes (\gamma-1)x.
\] 
If $m$ is sufficiently large, then 
\[
|(\gamma-1)y\otimes\gamma(x)|\leq\frac{1}{2c}|y\otimes x|, \quad |y\otimes (\gamma-1)x|\geq \frac{1}{c}|y\otimes x|.
\]
Thus $|(\gamma-1)(y\otimes x)|\geq \frac{1}{c}|y\otimes x|$. This yields the desired result.  Consequently, for $i\geq0$, we get
\[
H^i(\Gal(k_\infty/k), M\hat{\otimes}_{k_{m_0}}X_m)=0,
\]
by Hochschild-Serre spectral sequence. This proves the lemma.
\end{proof}
Now replacing the $A_{m_0}$-module $M(X)$ as in Proposition \ref{P: local p-adic Simpson}(P1) by $M(X)\otimes_{A_{m_0}}A_m$ if necessary, we can conclude using Lemma \ref{L: cohomology-descent} that
\[
H^i(\Gal(k_\infty/k), \mH(L)(j)(X_K))\simeq H^i(\Gal(k_\infty/k), M(X)(j))
\]
Clearly, $H^i(\Gal(k_\infty/k), M(X)(j))$ is a finite $A$-module, and vanishes if $|j|\gg 0$. Moreover, it vanishes if $i\geq2$ because $\Gal(k_\infty/k)$ has cohomological dimension $1$. 

By Proposition \ref{P: local p-adic Simpson}, $M(X)$ is compatible with standard \'etale base extensions. Note that standard \'etale morphisms are flat. Hence it is straightforward  (cf. \cite[Lemma 1.4.3]{Liu}) to see that $H^i(\Gal(k_\infty/k), M(X)(j))$ is compatible with standard \'etale base extensions. This yields (a) and (b), and therefore finishes the proof of Part (i) of the theorem.

\medskip

Next we prove Part (ii). 
First note that the map of period sheaves \eqref{E: str map I} induces a natural map
\begin{equation}\label{E: dR pullback}
f^*D^i_{\dR}(\bL)=f^*R^i\nu_{X,*}(\hat\bL\otimes\mO\bB_{\dR,X})\to D_{\dR}^i(f^*\bL)=R^i\nu_{Y,*}(\widehat{f^*\bL}\otimes\mO\bB_{\dR,Y})
\end{equation}
by a similar procedure as before. Namely, we have a similar map \eqref{E: str map III} with $\mO\bC$ replaced by $\mO\bB_{\dR}$ and a similar map \eqref{E: str map IV} with $\nu'$ replaced by $\nu$ and all functors replaced by their derived version. Then we obtain a derived version of \eqref{E: map1} in the current context. Note that since $D^i_{\dR}(\bL)=R^i\nu_{X,*}(\hat\bL\otimes\mO\bB_{\dR,X})$ is already locally free of finite rank, we can replace $Lf^*$ by $f^*$ and therefore obtain \eqref{E: dR pullback}. It remains to prove that it is an isomorphism.

Before proceeding, we introduce some notations. For $-\infty\leq a\leq b\leq \infty$, write
\[
\mR\mH^{[a,b]}(\bL)=\on{Fil}^a\mR\mH(\bL)/\on{Fil}^{b+1}\mR\mH(\bL).
\] 
Then $\nu'_*(\bL\otimes\mO\bB^{[a, b]}_{\dR})=\mR\mH^{[a,b]}(\bL)$ by virtue of Theorem \ref{T: p-adic RH}.
Without loss of generality, we may assume that $X=\Spa(A,A^+)$ and $Y=\Spa(B,B^+)$ are affinoid spaces admitting toric charts, and $\bL$ is a $\bZ_p$-local system. From the proof of Part (i) we have seen that under this situation there exist $a, b\in\bZ$ such that
\[
D_{\dR}^i(\bL)(X)=H^i(\Gal(k_\infty/k), \mR\mH^{[a,b]}(\bL)(X_K)),
\]
and
\[
D_{\dR}^i(f^*\bL)(Y)=H^i(\Gal(k_\infty/k),  \mR\mH^{[a,b]}(f^*\bL)(Y_K)).
\]
It follows that $H^i(\Gal(k_\infty/k),\mR\mH^{[a,b]}(\bL)(X_K))$ for $i=0,1$ are flat $A$-modules by Part (i). 
Note that in addition the standard complex computing 
\[
H^*(\Gal(k_\infty/k), \mR\mH^{[a,b]}(\bL)(X_K))
\] 
is a complex of flat $A$-modules (since $A_K$ is flat over $A$). This implies that 
\begin{equation}\label{E: base change}
H^i(\Gal(k_\infty/k), \mR\mH^{[a,b]}(\bL)(X_K))\otimes_AB\simeq H^i(\Gal(k_\infty/k),\mR\mH^{[a,b]}(\bL)(X_K)\otimes_AB).
\end{equation}
Therefore, to see that (\ref{E: dR pullback}) is an isomorphism, by (\ref{E: base change})  it is enough to apply the following result.

\begin{lem}\label{L: cohomology base change}
Let $X=\Spa(A,A^+)$ and $Y=\Spa(B,B^+)$ be smooth affinoid spaces admitting toric charts, and let $f: X\to Y$ be a morphism of rigid analytic varieties. Let $\bL$ be a $\bZ_p$-local system on $X_{\on{et}}$. Then for any $a,b\in\bZ$ and $i\geq0$,  the natural map
\[
H^i(\Gal(k_\infty/k),\mR\mH^{[a,b]}(\bL)(X_K)\otimes_AB)\to H^i(\Gal(k_\infty/k),\mR\mH^{[a,b]}(f^*\bL)(Y_K))
\]
is an isomorphism.
\end{lem}
\begin{proof}
It is enough to prove that for $j\in[a,b]$, the natural map
\[
H^i(\Gal(k_\infty/k),\mH(\bL)(X_K)(j)\otimes_AB)\to H^i(\Gal(k_\infty/k),\mH(f^*\bL)(Y_K)(j))
\]
is an isomorphism. Now let $M(X)$ be as in Proposition \ref{P: local p-adic Simpson} (P1) for the local system $\bL$. Then $M(X)$ is a finite projective $A_{m_0}$-module for some sufficiently large $m_0$ such that $\mH(\bL)(X_K)=M(X)\otimes_{A_{m_0}}A_K$. Let $M(Y)=M(X)\otimes_{A}B$. It follows that 
\[
M(Y)\otimes_{B_{m_0}}B_K=(M(X)\otimes_{A_{m_0}}A_K)\otimes_{A_K}B_K=\mH(\bL)(X_K)\otimes_{A_K}B_K=\mH(f^*\bL)(Y_K)
\] 
by Theorem \ref{T: p-adic Simpson}(iii). Now we apply Lemma \ref{L: cohomology-descent} to get some sufficiently large $m$, so that $\gamma-1$ is continuously invertible on both 
$$\mH(\bL)(X_K)(j)/(M(X)(j)\otimes_{k_{m_0}}k_m)$$ 
and 
$$\mH(f^*\bL)(Y_K)(j)/(M(Y)(j)\otimes_{k_{m_0}}k_m)$$
whenever $v_p(\chi(\gamma)-1)\geq m$. Thus $\gamma-1$ is continuously invertible on 
$$(\mH(\bL)(X_K)(j)\otimes_AB)/((M(X)(j)\otimes_{k_{m_0}}k_m)\otimes_AB).$$
 Therefore, we deduce that
 \begin{align*}
 H^i(\Gal(k_\infty/k), \mH(\bL)(X_K)(j)\otimes_AB))&=H^i(\Gal(k_\infty/k), (M(X)(j)\otimes_{k_{m_0}}k_m)\otimes_AB)\\
&=H^i(\Gal(k_\infty/k), M(Y)(j)\otimes_{k_{m_0}}k_m)\\
&=H^i(\Gal(k_\infty/k), \mH(f^*\bL)(Y_K)(j)). 
\end{align*}
\end{proof}

Next, we prove Part (iii). Again let $K=\widehat{k_\infty}$. The question is local, so we may assume that $X=\Spa(A,A^+)$ is an affinoid space admitting a toric chart. To proceed, first note that by assumption and Parts (i), (ii), $\bL_{\bar y}$ is de Rham for any classical point $y$. 
Now we fix a sufficiently large $b$ so that 
$R\nu_*(\hat\bL\otimes\gr^{b'}\mO\bB_{\dR})=0$ whenever $b'\geq b$. Now for every interval $[a,b]$, we have the six-term exact sequence
\begin{multline}\label{M: 6 terms}
0\to \nu_*(\hat\bL\otimes\mO\bB_{\dR}^{[a+1,b]})\to  \nu_*(\hat\bL\otimes\mO\bB_{\dR}^{[a,b]})\to \nu_*(\hat\bL\otimes\gr^a\mO\bB_{\dR})\to \\
R^1\nu_*(\hat\bL\otimes\mO\bB_{\dR}^{[a+1,b]})\to  R^1\nu_*(\hat\bL\otimes\mO\bB_{\dR}^{[a,b]})\to R^1\nu_*(\hat\bL\otimes\gr^a\mO\bB_{\dR})\to 0.
\end{multline}
We will show that each term in the exact sequence is a finite locally free $\mO_{X_{\on{et}}}$-module, and the connecting map $\nu_*(\hat\bL\otimes\gr^a\mO\bB_{\dR})\to 
R^1\nu_*(\hat\bL\otimes\mO\bB_{\dR}^{[a+1,b]})$ is zero. Note that for $a\ll 0$, $R\nu_*(\hat\bL\otimes \gr^a\mO\bB_{\dR})=0$ and 
\[
R\nu_*(\hat\bL\otimes\mO\bB_{\dR}^{[a+1,b]})=R\nu_*(\hat\bL\otimes\mO\bB_{\dR}^{[a,b]})
=R\nu_*(\hat\bL\otimes\mO\bB_{\dR})
\] 
Thus the claim holds automatically in this case. Therefore we may do induction on $a$ and assume that the claim holds for $a-1$.

For simplicity, write 
$M^{[a,b]}= \mR\mH^{[a,b]}(\bL)(X_K)$
and $M=\mH(\bL)(X_K)$.
Then instead of \eqref{M: 6 terms}, we may consider the following exact sequence of Galois cohomology
\begin{multline}\label{M: Gal}
0\to H^0(\Gal(k_\infty/k),M^{[a+1,b]})\to  H^0(\Gal(k_\infty/k),M^{[a,b]})\to  H^0(\Gal(k_\infty/k),M(a))\to \\
 H^1(\Gal(k_\infty/k),M^{[a+1,b]})\to   H^1(\Gal(k_\infty/k),M^{[a,b]})\to H^1(\Gal(k_\infty/k),M(a))\to 0.
\end{multline}
Note that taking $H^1$ always commutes with base change to a classical point $y$.
Therefore, in the following commutative diagram of the connecting maps
\[\begin{CD}
H^0(\Gal(k_\infty/k),M(a))\otimes k(y)@>>>  H^1(\Gal(k_\infty/k),M^{[a+1,b]})\otimes k(y) \\
@VVV@VV\simeq V\\
H^0(\Gal(k_\infty/k),M(a)\otimes k(y)) @>>>  H^1(\Gal(k_\infty/k),M^{[a+1,b]}\otimes k(y)),
\end{CD}\]
the right vertical arrow is an isomorphism.
Note that by Lemma \ref{L: cohomology base change} and Remark \ref{R: point p-adic Simpson}, the bottom line coincides with
\[H^0(\Gal(k_\infty/k), (\bL'_{\bar y}\otimes \hat{\bar k}(a))^{\Gal(\hat{\bar k}/k_\infty)})\to H^1(\Gal(k_\infty/k), (\bL'_{\bar y}\otimes \mathrm{B}_{\dR}^{[a+1,b]})^{\Gal(\hat{\bar k}/k_\infty)})\]
with $\bL'_{\bar y}=\Ind_{\Gal(\bar k/k(y))}^{\Gal(\bar k/k)}\bL_{\bar y}$. Here  $\mathrm{B}_{\dR}$ stands for Fontaine's de Rham period ring (rather than $\bB_\dR(K,\mO_K)$ considered in \S~\ref{S: p-adic RH theorem}). Since $\bL_{\bar y}$ is de Rham,  $\bL'_{\bar y}$ is de Rham as well. Thus it is a zero map.  
So the map in the top row is also zero. Since this is true that every classical point, the connecting map 
\[
H^0(\Gal(k_\infty/k),M(a))\to
 H^1(\Gal(k_\infty/k),M^{[a+1,b]})
 \] 
 must be zero. In particular, \eqref{M: Gal} breaks into two short exact sequences 
 \begin{equation}\label{E: R0}
0\to H^0(\Gal(k_\infty/k),M^{[a+1,b]})\to  H^0(\Gal(k_\infty/k),M^{[a,b]})\to  H^0(\Gal(k_\infty/k),M(a))\to 0
 \end{equation} 
 and 
 \begin{equation}\label{E: R1}
 0\to  H^1(\Gal(k_\infty/k),M^{[a+1,b]})\to   H^1(\Gal(k_\infty/k),M^{[a,b]})\to H^1(\Gal(k_\infty/k),M(a))\to 0.
 \end{equation}
Similarly, using Lemma \ref{L: cohomology base change}  and Remark \ref{R: point p-adic Simpson} again, the base change of (\ref{E: R1}) to $y$ coincides with the sequence
\begin{multline}
0\to H^1(\Gal(k_\infty/k), (\bL'_{\bar y}\otimes \mathrm{B}_{\dR}^{[a+1,b]})^{\Gal(\hat{\bar{k}}/k_\infty)})\to \\ H^1(\Gal(k_\infty/k),(\bL'_{\bar y}\otimes \mathrm{B}_{\dR}^{[a,b]})^{\Gal(\hat{\bar{k}}/k_\infty)})
\to H^1(\Gal(k_\infty/k), (\bL'_{\bar y}\otimes \hat{\bar k}(a))^{\Gal(\hat{\bar{k}}/k_\infty)}\to 0,
\end{multline}
which remains exact because $\bL'_{\bar y}$ is de Rham.
Since $H^1(\Gal(k_\infty/k),M^{[a,b]})$ is projective over $A$ by induction, we deduce that $H^1(\Gal(k_\infty/k),M(a))$ is flat over $A$, and therefore projective over $A$. Thus $H^1(\Gal(k_\infty/k),M^{[a+1,b]})$ is a projective $A$-module as well. 

For terms in (\ref{E: R0}), first note that the cohomology of $M(a)$ is computed by the complex 
\[
M(a)^{\Delta}\stackrel{\gamma-1}\to M(a)^{\Delta},
\]
where $\Delta$ is the torsion subgroup of $\Gal(k_\infty/k)$, and $\gamma$ is a topological generator of $\Gal(k_\infty/k)/\Delta$. Since $\Delta$ is a finite group and $M$ is flat over $A$, $M^{\Delta}$ is flat over $A$ as well. Thus the flatness of $H^1(\Gal(k_\infty/k), M(a))$ ensures the flatness of $H^0(\Gal(k_\infty/k), M(a))$. Thus $H^0(\Gal(k_\infty/k),M^{[a+1,b]})$ is flat as well since $H^0(\Gal(k_\infty/k),M^{[a,b]})$ is flat by induction. 

\medskip

Let us record the following corollary of the proof.
\begin{cor}\label{C: DdR and RH}
Assume that $\bL$ is a de Rham local system on $X$. Then 
\begin{enumerate}
\item[(i)] $\bL$ is Hodge-Tate in the sense of Remark \ref{R: HT local system};

\item[(ii)] 
$(\mR\mH(\bL),\nabla_\bL)\simeq (D_\dR(\bL),\nabla_\bL)\hat\otimes_k\B_\dR$.
\end{enumerate}
\end{cor}

\begin{rmk}
Note that Part (ii) of the above corollary together with Theorem \ref{T: p-adic RH} recover \cite[Theorem 1.10]{Sch2}.
\end{rmk}

It remains to prove Parts (iv) and (v). 
To prove that $\bL$ is a de Rham local system in the sense of \cite{Sch2}, we need to show that the natural map
\[\nu^*D_{\dR}(\bL)\otimes_{\mO_X} \mO\bB_{\dR}\to \hat\bL\otimes\mO\bB_{\dR}\]
is an isomorphism compatible with connection and filtration. But this follows from Part (ii) of the above corollary and Theorem \ref{T: p-adic RH} (iii). 

The remaining statements are clear.
\end{proof}

\section{Applications}\label{S: application}
In this section, we give some applications of our results. In particular, we prove Theorem \ref{Main cor} and Theorem \ref{C: application}.

\subsection{Rigidity of geometric $p$-adic representations}
Having established Theorem \ref{Main thm}, to prove Theorem \ref{Main cor} it remains to show the following proposition.

\begin{prop}\label{L: purity}
Let $X$ be a geometrically connected algebraic variety over a number field $E$ and let $\bL$ be a rank $n$ $\bZ_p$-local system on $X$. If for some point $x\in X(F)$ where $F$ is a finite extension of $E$, the Galois representation $\bL_{\bar x}$ is unramified almost everywhere, then for every  point $\bar y\in X(\bar E)$, $\bL_{\bar y}$ is unramified almost everywhere (as a Galois representation of the field of the definition of $\bar y$).
\end{prop}
\begin{proof}
Using resolution of singularities,  we may assume that $X$ is smooth.
Let us choose a smooth projective compactification $\bar X$ of $X$ and let $D=\bar X- X$ denote the boundary divisor. We choose $N$ large enough such that
\begin{itemize}
\item we can spread $\bar X$ out as a smooth projective scheme $\overline\frakX$ over $\mO_E[1/N]$, where $\mO_E$ is the ring of integers of $E$;
\item the extension $\mO_F[1/N]/\mO_E[1/N]$ is unramified;
\item the order $|\GL_n(\bF_p)|=p^{n(n-1)/2}(p-1)(p^2-1)\cdots(p^n-1)$ divides $N$. 
\end{itemize}
Now for a place $v$ of $\mO_E[1/N]$, let $\overline\frakX_v$ denote the special fiber of $\overline\frakX$ at $v$. Let $s_v$ be the generic point of $\overline\frakX_v$, and $\eta_v$ be the generic point of the henselization $\overline\frakX_{(s_v)}$ of $\overline\frakX$ at $s_v$. We have the following standard exact sequence of Galois groups for the trait $(\overline\frakX_{(s_v)},s_v,\eta_v)$
\begin{equation}\label{E: inertia group}
1\to J_v\to \Gal(\overline{\eta}_v/\eta_v)\to \Gal(\overline{s}_v/s_v)\to 1.
\end{equation}
Let $J_v^{\on{t}}$ denote the tame quotient of $J_v$. Since $\overline\frakX\to \mO_E[1/N]$ is smooth, it induces an isomorphism from $J_v^{\on{t}}$ to the tame inertia $I_v^{\on{t}}$ of $E_v$ (by choosing a map from $\overline\eta_v$ to $\Spec\ \overline{E_v^h}$, where $E_v^h$ is the fractional field of the henselization of $\mO_E$ at $v$). 

Let $\frakD$ denote the Zariski closure of $D$ in $\overline\frakX$. The  the $\bZ/p^m$-local system $\bL/p^m$ defines a finite \'etale cover $X_m\to X$ with Galois group $\GL_n(\bZ/p^m)$. Let $\overline\frakX_m$ denote the normalization of $\overline\frakX$ in the fractional field of $X_m$. By the Zariski--Nagata purity, the branch locus of $\overline\frakX_m\to \overline\frakX$ is a divisor $\frakD_m$ on $\overline \frakX$, pure of codimension one. So 
$$\frakD_m=\sum a_v\overline\frakX_v+B_m$$
where $\overline\frakX_v$ is regarded as a vertical divisor and $B_m$ is a horizontal divisor supported in $\frakD$. Let $S_m$ be the set of places $v$ of $\mO_E[1/N]$ such that $a_v\neq 0$. Here is the lemma.

\begin{lem}The set $S=\cup_{m\geq 1} S_m$ is finite.
\end{lem}
\begin{proof}
By the valuative criterion, the point $x:\Spec\ F\to X$ extends to an arithmetic curve $\Spec(\mO_F[1/N])\to \overline\frakX$, still denoted by $x$. 
Let $T\subset S$ be the subset consisting of those $v$ such that $x(\Spec\ \mO_F[1/N])$ intersects with $\overline\frakX_v$ outside the divisor $\frakD_v$. Note that $S\setminus T$ is finite.
We claim that the Galois representation $\Gal(\overline F/F)\to \GL( \bL_{\bar x})$ is ramified above $v\in T$. Since it is unramified almost everywhere, it implies that $T$ (and therefore $S$) is finite.

Now let $v\in S_m\cap T$. Note that the residue characteristic of $v$ is coprime to $N$. Therefore, by our assumption on $N$, the composition of the maps 
$$J_v\to \Gal(\overline{\eta}_v/\eta_v)\to \pi_1(\overline\frakX\times E_v^h,\overline\eta_v)\to \GL_n(\bZ/p^m\bZ)$$ factors through the tame inertia 
\begin{equation}\label{E: ramified}
J_v^{\on{t}}\to \GL_n(\bZ/p^m\bZ).
\end{equation} 
By our assumption, this is a non-trivial map. Let $w$ be a place of $\mO_F[1/N]$ lying over $v$. %Let $E_v^h$ (resp. $F_w^h$) denote the fractional field of the henselization of $\mO_E$ at $v$ (resp. henselization of $\mO_F$ at $w$). 
Note that we have the following commutative diagram

\[\begin{CD}
I_w^{\on{t}}@<<< I_w @>>>\Gal(\overline{F}_w/F_w)\\
@VVV@VVV@VVV\\
J_v^{\on{t}}@<<<J_v @>>> \pi_1(\overline{\frakX}\times E^h_v, \overline{\eta}_v)\\
\end{CD}\]
and the left vertical map is an isomorphism (since both map isomorphically to $I_v^{\on{t}}$). It follows that
the Galois representation $\bL_{\bar x}$ is ramified at $w$.
\end{proof}
Now every $y:\Spec\ F'\to X$ extends to an arithmetic curve $y: \Spec\ \mO_{F'}[1/N]\to\overline\frakX$ that can only intersect with $\frakD$ at finite many points. Therefore, away from the places underlying these points, the places dividing $N$ and the places in $S$, $\bL_{\bar y}$ is unramified.
\end{proof}

\subsection{An application to Shimura varieties}\label{S: Shimura}
Now we turn to Shimura varieties.
As usual, for a number field $E$, let $E^{\on{ab}}$ denote its maximal abelian extension in (a fixed) algebraic closure $\overline E$.
Let $\bA_E$ (resp. $\bA_{E,f}$) denote the ad\`eles (resp. finite ad\`eles) of $E$. If $E=\bQ$, we drop $E$ from the subscripts of these notations. 

Let $T$ be a $\bQ$-torus and $\mu: \bG_m\to T$ be a cocharacter defined over a number field $F$, i.e. an $F$-homomorphism $\mu: \bG_m\to T_F$. It then induces a homomorphism between $\bQ$-tori
$$\on{N}\mu: \Res_{F/\bQ}\bG_m\longrightarrow \Res_{F/\bQ}T_F\stackrel{\mathrm{Nm}}{\longrightarrow} T,$$
where the second map is the usual norm map. 
\begin{ex}\label{reflex norm}
Let $(E,\Phi)$ be a CM type. It gives rise to a $\bQ$-torus $\Res_{E/\bQ}\bG_m$ whose cocharacter group is the free abelian group with a basis $[\tau]$ labelled by embeddings $\tau:E\to \bC$,  and a cocharacter $\mu_\Phi=\sum_{\tau\in\Phi}[\tau]$ defined over the reflex field $E^*$. Then 
$$\on{N}\mu_{\Phi}:\Res_{E^*/\bQ}\bG_m\to \Res_{E/\bQ}\bG_m$$ is usually called the reflex norm (\cite[\S 11]{M2}).
\end{ex}

Let $K\subset T(\bA_f)$ be an open compact subgroup. Recall that the double coset $T(\bQ)\backslash T(\bA_f)/K=T(\bQ)\backslash T(\bA)/K T(\bR)$ is always finite and therefore the homomorphism
\[F^\times\backslash \bA_F^\times\stackrel{\on{N}\mu}{\longrightarrow} T(\bQ)\backslash T(\bA)\to T(\bQ)\backslash T(\bA_f)/K\]
factors through 
\begin{equation}\label{reflex norm}
F^\times\backslash \bA_F^\times\stackrel{\on{Art}_F}{\longrightarrow} \Gal(F^{\on{ab}}/F)\stackrel{r(\mu)_K}{\longrightarrow} T(\bQ)\backslash T(\bA_f)/K,
\end{equation}
where the first map is the global Artin map, normalized such that for every finite place $v$ of $F$ and a uniformizer $\pi_v$ of $F_v\subset \bA_F$, its image in $\Gal(F^{\on{ab}}_v/F_v)\subset \Gal(F^{\on{ab}}/F)$ projects to the \emph{geometric} Frobenius. Taking the inverse limit over all open compact subgroups $K\subset T(\bA_f)$, we obtain
\[r(\mu): \Gal(F^{\on{ab}}/F)\to T(\bQ)^-\backslash T(\bA_f),\]
where $T(\bQ)^-=\underleftarrow{\lim}_{K} T(\bQ)/(T(\bQ)\cap K)$ denotes the closure of $T(\bQ)$ in $T(\bA_f)$.

Let $F_K/F$ be the finite abelian extension inside $F^{\on{ab}}$ such that $\Gal(F^{\on{ab}}/F_K)$ is the kernel of the map $r(\mu)_K$ in \eqref{reflex norm}.
By definition, $r(\mu)_K$ restricts to a continuous homomorphism
$$\Gal(F^{\on{ab}}/F_K)\to K/(K\cap T(\bQ)^-),$$
still denoted by $r(\mu)_K$.
Now, let 
$\rho: T\to \GL(V)$ 
be a $\bQ$-rational representation of $T$ satisfying
\begin{equation}\label{E: triv}
\rho|_{T(\bQ)^-\cap K}=1.
\end{equation} 
It induces a Galois representation
\[r(\mu,\rho)_K:\Gal(F^{\on{ab}}/F_K)\to K/(K\cap T(\bQ)^-)\to \GL(V\otimes\bA_f).\]
Let $p$ be a finite prime. By projecting to the $p$-component, we obtain a $p$-adic representation
\[r(\mu,\rho)_{K,p}:\Gal(F^{\on{ab}}/F_K)\to \GL(V\otimes\bQ_p).\]
\begin{lem}\label{cryrep}
The $p$-adic representation $r(\mu,\rho)_{K,p}$ is unramified almost everywhere and is potentially crystalline at every place $v$ of $F_K$ above $p$.
\end{lem}
\begin{proof}
It is clear that the representation $r(\mu,\rho)_{K,p}$ is unramified almost everywhere since it factors as 
$\Gal(F^{\on{ab}}/F_K)\to K/(K\cap T(\bQ)^-)K^p$, where $K^p$ is the prime-to-$p$ component of $K$. We prove that it is potentially crystalline.

Let $L\subset \overline{F}$ be a splitting field of $T$, containing all the embeddings of $F_K$ to $\overline F$. Note that $\rho\otimes L$ splits into $1$-dimensional characters of $T_L$. Therefore, we may assume that $\rho=\chi|_{T}$ where $\chi$ is a character of $T_L$ and $T\subset \Res_{L/\bQ}T_L$ is the natural inclusion. In addition, we can choose a place $\la$ of $L$ above $p$. It is enough to show that the induced $\la$-adic representation
$$r(\mu,\rho)_{K,\la}:\Gal(F^{\on{ab}}/F_K)\to L_\la^\times$$
is potentially crystalline.  

Note that the restriction of $r(\mu,\rho)_{K,\la}$ to $F_{K,v}^\times\subset \Gal(F_{K,v}^{\on{ab}}/F_{K,v})\subset \Gal(F^{\on{ab}}/F_K)$ is given by
\[F_{K,v}^\times\stackrel{\mu}{\to} T(F_{K,v})\stackrel{\mathrm{Nm}}{\to} T(\bQ_p) \stackrel{\chi}{\to} L_\la^\times, \]
which can be written as $\prod \tau\circ [a_\tau]:F_{K,v}^\times\to L_\la^\times$, where $\tau$ ranges over all embeddings from $F_{K,v}$ to $L_\la$, $a_\tau\in\bZ$, and $[n]: F_{K,v}^\times\to F_{K,v}^\times$ is the $n$th power map.
Therefore, on a finite index subgroup of $\Gal(F_{K,v}^{\on{ab}}/F_{K,v})$,  
$$r(\mu,\rho)_{K,\la}=\prod_{\tau} \tau\circ \on{LT}^{a_\tau},$$ where $\on{LT}: \Gal(F_{K,v}^{\on{ab}}/F_{K,v})\to \mO_{F_{K,v}}^\times$ is the Lubin-Tate character.
It is well-known that such a representation is crystalline (cf. \cite[Proposition B.4]{Con}). The lemma follows.
\end{proof}

We make a digression to discuss Condition \eqref{E: triv}.
For a $\bQ$-torus $T$, let $T^a$ be the maximal anisotropic subtorus over $\bQ$, i.e. $T^a$ is the neutral connected component of the intersection of all kernels of $\bQ$-rational characters of $T$. Let
$T^s\subset T^a$ be the maximal subtorus of $T^a$ that is $\bR$-split. Let $T^c=T/T^s$.
\begin{lem} Let $\rho$ be an algebraic representation of $T$.
\begin{enumerate}
\item[(i)]  If $\rho$ is trivial on $T(\bQ)^-\cap K$ for some open compact subgroup $K\subset T(\bA_f)$, then $\rho|_{T^s}$ is trivial.
\item[(ii)] If $\rho|_{T^s}$ is trivial, then for some small enough open compact subgroup $K\subset T(\bA_f)$, $\rho|_{T(\bQ)^-\cap K}$ is trivial.
\end{enumerate}
\end{lem}
\begin{proof}
Note that $T^s(\bQ)\backslash T^s(\bA)$ is compact since $T^s$ is anisotropic. Therefore $(K\cap T^s(\bQ))\backslash T^s(\bR)$ is also compact. Since $T^s(\bR)\simeq (\bR^\times)^r$, $K\cap T^s(\bQ)$ must contain a lattice of rank $r$, and therefore is Zariski dense in $T^s$.
Therefore, \eqref{E: triv} implies that $\rho|_{T^s}=1$.

On the other hand, $(T^c)^s=1$ implies that $T^c(\bQ)$ is discrete in $T^c(\bA_f)$ (e.g. \cite[Theorem 5.26]{M2}). Therefore for $K$ small enough, $\pi(K)\cap T^c(\bQ)$ is trivial where $\pi:T\to T^c$ is the projection, giving (ii).
\end{proof}

\begin{ex}\label{CMAV}
Let $(E,\Phi)$ be a CM type, and let $E^*$ be the reflex field. Let $A$ be an abelian variety defined over a number field $F$ over $E^*$ such that $\End (A)\otimes \bQ\simeq E$ with $\Phi$ the induced CM type. 
Let $E_0\subset E$ be the maximal totally real subfield, and let $T\subset \Res_{E/\bQ}\bG_m$ be the subtorus which is the preimage of $\bG_m$ under the norm map $\Res_{E/\bQ}\bG_m\to \Res_{E_0/\bQ}\bG_m$. Then $T=T^c$ and $\mu_\Phi$ is a cocharacter of $T$.
Let $V=E$ and let $\rho$ be the natural representation 
$$\rho:T\subset \Res_{E/\bQ}\bG_m\simeq \GL_1(E)\subset \GL(V).$$
Then the theory of complex multiplication implies that the above representation $r(\mu_\Phi,\rho)_{K,p}$ appears as the Galois representation of the rational $p$-adic Tate module of $A$.
\end{ex}

Now let $(G,X)$ be a Shimura datum. This means that $G$ is a reductive group defined over $\bQ$, $X$ is a $G(\bR)$-conjugacy class of homomorphisms $$h: \bS=\Res_{\bC/\bR}\bG_m\to G_\bR$$ satisfying Deligne's axioms (SV1)-(SV3) \cite[Definition 5.5]{M2}
and the quotient
\[G(\bQ)\backslash X\times G(\bA_f)/K\]
is the set of $\bC$-points of a quasi-projective algebraic variety $\on{Sh}_K(G,X)$ over $\bC$ (\cite[\S 5]{M2}).

We recall the definition of the reflex field $E=E(G,X)$. 
Let us fix the isomorphism $\bS_\bC=\bG_m\times\bG_m$ given as follows:
for any $\bC$-algebra $R$,  
$$(\bC\otimes_\bR R)^\times\simeq R^\times\times R^\times,\quad z\otimes r\mapsto (zr, \bar{z}r).$$ 
Let $\bG_m\to \bS_\bC$ be the inclusion of the first factor.
Then the map $h$ induces a homomorphism (usually called the Hodge cocharacter)
\begin{equation}\label{JL:coch}
\mu_h: \bG_m\to \bS_\bC\to G_{\bC}.
\end{equation}
As $h$ varies in $X$, $\mu_h$ form a conjugacy class of 1-parameter subgroups of $G$ over $\bC$. Then $E(G,X)$ is the field of definition of this conjugacy class.
In particular, if $G=T$ is a torus, then $X=\{h\}$ is a single homomorphism and $E(T,\{h\})$ is the field of the definition of the cocharacter $\mu_h$.

Next we recall the definition of the canonical model of $\on{Sh}_K(G, X)$, i.e. the unique descent from $\bC$ to $E=E(G,X)$ of $\on{Sh}_K(G, X)$
subject to certain properties.
We first assume that $G=T$ is a torus. Then $T(\bQ)\backslash T(\bA_f)/K$ is a finite set. To describe it as an algebraic variety over $E$ is equivalent to describing the action of $\Gal(\overline E/E)$ on this set. But this is defined as the action of $\Gal(E^{\on{ab}}/E)$ via $r(\mu_h)_K$ with the natural multiplication of $T(\bQ)\backslash T(\bA_f)/K$ on itself.

For general $G$, recall that a point $h\in X$ is called special if there exists a $\bQ$-torus $T\subset G$ such that $h$ factors as $h:\bS\to T_\bR\to G_\bR$. Then such a point gives an embedding of Shimura data $(T,\{h\})\to (G,X)$. Let $E(h)$ be the field of the definition of $\mu_h$. It follows that $E=E(G,X)\subset E(h)$.
Then the canonical model is the descent from $\bC$ to $E$ of $\on{Sh}_K(G,X)$
equipped with the $G(\bA_f)$-action,
such that for every special point $x\in X$, the natural map
\[T(\bQ)\backslash \{h\}\times T(\bA_f)/(K\cap T(\bA_f))\to G(\bQ)\backslash X\times G(\bA_f)/K \]
is defined over $E(h)$. It is known that the canonical model exists and is unique (e.g. see \cite{M}).
By abuse of notation,  let $\on{Sh}_K(G,X)$ denote the canonical model in the rest of the paper.

Following \cite{M}, let $Z_G^s$ denote the maximal anisotropic subtorus of $Z_G$ that is \emph{split} over $\bR$, and let $G^c=G/Z_G^s$. 

Now let $\rho: G^c\to \GL(V)$ be a $\bQ$-rational representation of $G^c$. Then for $K$ sufficiently small, it gives a Betti local system on $\on{Sh}_K(G,X)(\bC)$ as
\[\bL_V: =V\times^{G(\bQ)} (X\times G(\bA_f))/K,\]
which, after tensoring with $\bQ_p$, in fact descends to an \'etale local system $\bL_{V,p}$ over $\on{Sh}_K(G,X)$. We recall the construction of $\bL_{V,p}$ in the below (see also \cite[\S III.6]{M}).

Write $K=K_pK^p$ where $K_p\subset G(\bQ_p)$ and $K^p\subset G(\bA_f^p)$, and consider the representation
\[\rho: K_p\to G(\bQ_p)\to \GL(V_{\bQ_p}).\]
As $K_p$ is compact, we can find a lattice $\Lambda_p\subset V_{\bQ_p}$ fixed by $K_p$. Let
\[K_p^{(n)}=K_p\cap \rho^{-1}(\{g\in \GL(\Lambda_p)\mid g\equiv1 \mod p^n\}).\]
Then $\{K_p^{(n)}\}_{n\geq1}$ form a system of open neighborhoods of $G(\bQ_p)$. Note that we have a representation 
\[
\bar\rho: K_p/K_p^{(n)}\to \GL(\Lambda_p/p^n),
\]
and $\on{Sh}_{K_p^{(n)}K^p}(G,X)\to \on{Sk}_{K_pK^p}(G,X)$ is finite \'etale with the deck transformation group $K_p/K_p^{(n)}$. Therefore, we obtain an \'etale $\bZ/p^n\bZ$-local system on $\on{Sh}_{K_pK^p}(G,X)(\bC)$ as
\[\bL_{V,p,n}=\on{Sh}_{K_p^{(n)}K^p}(G,X)(\bC)\times^{K_p/K_p^{(n)}}(\Lambda_p/p^n).\]
Then
\[\bL_{V,p}=(\limproj_{n} \bL_{V,p,n})\otimes\bQ\]
is an \'etale $\bQ_p$-local system on $\on{Sh}_{K_pK^p}(G,X)(\bC)$. It is straightforward to check that $\bL_{V,p}$ only depends on $V_{\bQ_p}$ as a $\bQ_p$-representation of $G^c$ and is independent of the choice of the lattice $\Lambda_p$.
Note that since all the $\on{Sh}_{K_p^nK^p}(G,X)$'s are defined over $E$, $\bL_{V,p}$ also descends to $E$. Therefore, $V\mapsto \bL_{V,p}$ is a well-defined tensor functor from the category  $\on{Rep}_{\bQ_p}(G^c)$ of $\bQ_p$-rational representations of $G^c$ to the category of \'etale $\bQ_p$-local systems on $\on{Sh}_{K}(G,X)$.

\begin{ex}
If $(G,X)=(\on{GSp}(V), S^\pm)$ is the Siegel Shimura datum, then $\bL_V$ is the Betti local system of the first de Rham homology of the universal abelian scheme $A$, and $\bL_{V,p}$ is the $p$-adic Tate module of $A$.
\end{ex}

The following lemma is clear from the construction.

\begin{lem}\label{special}
Let $x=[h,a]_K\in \on{Sh}_K(G,X)$ be a special point and let $T_h\subset G$ be a $\bQ$-subtorus containing $h(\bS)$. Let $\rho'$ denote the restriction to  $T_h$ of the rational representation $V$.
Then the $p$-adic representation $\bL_{V,p}$ at $x$ is $r(\mu_h,\rho')_{K, p}$.
\end{lem}

Now Theorem \ref{C: application} follows from Theorem \ref{Main cor}, Lemma \ref{cryrep}, Lemma \ref{special} and the well-known fact that the set of special points is non-empty (in fact dense) in every connected component of $\on{Sh}_K(G, X)$ (e.g. see \cite[Lemma 13.5]{M2}).

We also have the following corollary, by Theorem \ref{T: p-adic RH for de Rham}. Let $v$ be a place of $E$ above $p$, and let $\on{Sh}_K(G,X)_{E_v}^{\on{ad}}$ denote the adic space associated to the base change of $\on{Sh}_K(G,X)$ to $E_v$.
\begin{cor}The functor $V\mapsto D_{\mathrm{dR}}(\bL_{V,p})$ is a tensor functor
from the category  $\on{Rep}_{\bQ_p}(G^c)$ to the category of vector bundles on $\on{Sh}_K(G,X)_{E_v}^{\on{ad}}$, and therefore defines a $G^c$-torsor $\mE$ on $\on{Sh}_K(G,X)_{E_v}^{\on{ad}}$.\end{cor}
\begin{rmk}
(i) The conjugacy class of Hodge cocharacters $\{\mu_h\}$ defines a conjugacy class $\mP$ of parabolic subgroups of $G^c$. The $G^c$-torsor $\mE$ then defines an associated bundle $\mE_\mP\to \on{Sh}_K(G,X)^\ad_{E_v}$ with fibers isomorphic to $\mP^\ad_{E_v}$. Note that since $V\mapsto D_{\mathrm{dR}}(\bL_{V,p})$ is in fact a tensor functor from $\on{Rep}_{\bQ_p}(G^c)$ to filtered vector bundles on $\on{Sh}_K(G,X)_{E_v}^{\on{ad}}$, it defines a section of $\mE_\mP$.

(ii) The $G^c$-torsor $\mE$ should be the base change of the standard principal $G^c$-torsor on $\on{Sh}_K(G,X)$ (\cite[\S III.4]{M}). Concretely, if $V$ is a $\bQ$-rational representation of $G^c$, then there is a vector bundle with a flat connection on $\on{Sh}_K(G,X)$ whose base change to $\bC$ corresponds to $\bL_V$ under the classical Riemann-Hilbert correspondence (\cite[\S III.6]{M}).
It is natural to conjecture that $D_{\dR}(\bL_{V,p})$ is the analytification of the base change to $E_v$ of this vector bundle with connection.

(iii) Let $x$ be a classical point of $\on{Sh}_K(G,X)_{E_v}^{\on{ad}}$. Then $D_{\dR}((\bL_{V,p})_{\bar x})\otimes \B_{\dR}^+$ provides a $\B_{\dR}^+$-lattice in $(\bL_{V,p})_{\bar x}\otimes \B_{\dR}$ as $(\bL_{V,p})_{\bar x}$ is de Rham. Therefore by Fargues' work on Breuil-Kisin modules over $\bA_{\on{inf}}$ (\cite{SW}) and the Tannakian formalism, one can associate to $x$ an isocrystal with additional structure $b_{x}\in B(G_{\bQ_p})$ (\cite{Ko}),  which should be the crystalline realization of  the mod $p$ fiber of the motive parameterised by $x$ (say the motive has a good reduction).
\end{rmk}

\begin{rmk}As a final remark, we point out that  one can use Theorem \ref{Main thm} to give an alternative proof of Fontaine's $C_{dR}$-conjecture for all abelian varieties and all smooth hypersurfaces in $\bP^n$ simultaneously. Namely, for every $p$, we choose a quadratic imaginary field $K$ such that $p$ splits in $K$. Let $E$ be an elliptic curve with CM by $K$. Then its $p$-adic Tate module is isomorphic to $\bQ_p\oplus \bQ_p(1)$ when restricted to the inertia and therefore is de Rham.
Note that  $E^g$ gives rise to a point in an appropriate moduli space of $g$-dimensional abelian varieties. Then we apply Theorem \ref{Main thm} to conclude the abelian variety case. Using the fact that the motives of the Fermat hypersurfaces appear in the motives generated by abelian varieties, $C_{dR}$ holds for them and therefore also holds for all hypersurfaces, again by Theorem \ref{Main thm}.
\end{rmk}

\end{document}